\documentclass[3p]{CSP}
\usepackage{amssymb}
\usepackage{changepage}

\usepackage{amsmath,amsbsy,amsfonts}
\usepackage{bm}

\usepackage{booktabs}	
\usepackage[colorlinks]{hyperref}
\usepackage[]{caption}

\newtheorem{theorem}{Theorem}[section]
\newtheorem{lemma}[theorem]{Lemma}
\newtheorem{corollary}[theorem]{Corollary}
\newtheorem{definition}{Definition}[section]

\newdefinition{example}{Example}[section]
\newenvironment{proof}[1][Proof]{\noindent\textbf{#1. }}{\hfill $\Box$}

\numberwithin{equation}{section}

\newcommand*{\norm}[1]{\lVert #1 \rVert}

\begin{document}

\begin{frontmatter}

\title{On high-order schemes for tempered fractional partial differential equations}

\author[mymainaddress,mysecondaddress]{Linlin Bu}
\ead{bulinlinlin@126.com}
\author[mysecondaddress,mythirdaddress]{Cornelis W. Oosterlee}
\ead{c.w.oosterlee@cwi.nl}
\address[mymainaddress]{School of Mathematics and Statistics, Xi'an Jiaotong University, Xi'an, Shaanxi 710049, P.R. China}

\address[mysecondaddress]{Applied Mathematics (DIAM), Delft University of Technology, Delft, The Netherlands}
\address[mythirdaddress]{Centrum Wiskunde $\&$ Informatica, Amsterdam, The Netherlands}

\begin{abstract}\rm
\begin{adjustwidth}{2cm}{2cm}{\itshape\textbf{Abstract:}}
In this paper, we propose third-order semi-discretized schemes in space based on the tempered weighted and shifted Gr\"unwald difference (tempered-WSGD) operators for the tempered fractional diffusion equation. We also show stability and convergence analysis for the fully discrete scheme based a Crank--Nicolson scheme in time. A third-order scheme for the tempered Black--Scholes equation is also proposed and tested numerically. Some numerical experiments are carried out to confirm accuracy and effectiveness of these proposed methods.
\end{adjustwidth}
\end{abstract}

\end{frontmatter}

\section{Introduction}
\label{}
\par
Nowadays, differential equations with fractional operators are often encountered, in a variety of science and engineering fields, such as in physics~\cite{HU2020109540,2017Multigrid,Liu2004Numerical,Gorenflo2007Discrete,Wang2012A,Wang2019A,Linlin2019Stable}, finance~\cite{Wyss2000The,Jumarie2010Derivation,Liang2010Option,WENTING2014ANALYTICALLY}, and biology~\cite{Magin2004Fractional,Jeon2012Anomalous}.
Mathematically, fractional calculus concerns time-space coupled operators.
These derivatives are useful mathematical tools to describe memory properties and hereditary effects.
The fractional derivatives in time are concerned, for example, with particle sticking and trapping, while the spatial versions can be used to model long particle jumps.

\par {\em Tempered} fractional derivatives are also often used as spatial operators. The Tempered fractional operators are introduced, for example, to describe the probability density functions for the positions of particles by exponentially tempering the probability of large jumps of L\'evy flight. Their definition is very similar to the fractional calculus in \cite{Chen2015Dis}, however, they were introduced with a different background. Tempered derivatives have been widely applied in physics \cite{High2016Li,BAEUMER20102438}, finance \cite{Cartea2006Fractional,Zhang2016The} and also in ground water hydrology \cite{Meerschaert2008Tempered}.

The well-known Riemann-Liouville fractional derivatives are defined in ~\cite{Fractional1993}, as follows. For $\alpha\in(n-1,n)$, let $u(x)$ be $(n-1)$-times continuously differentiable on interval $(a,b)$ and its $n$-times derivative be integrable on any subinterval of $[a,b]$. Then,
  the {\em left} Riemann-Liouville fractional derivative of order $\alpha$ is defined as
  \begin{equation}
  _{a}D_x^{\alpha}u(x)=\frac{1}{\Gamma(n-\alpha)}\frac{d^n}{dx^n}\int_a^x \frac{u(\xi)}{(x-\xi)^{\alpha-n+1}}d\xi,
  \end{equation}
 and the {\em right} Riemann-Liouville fractional derivative of order $\alpha$ is defined as
  \begin{equation}
  _xD_b^{\alpha}u(x)=\frac{(-1)^n}{\Gamma(n-\alpha)}\frac{d^n}{dx^n}\int_x^b \frac{u(\xi)}{(\xi-x)^{\alpha-n+1}}d\xi,
  \end{equation}
where $\Gamma$ represents the gamma function.

The Riemann-Liouville {\em tempered} fractional derivatives were defined and employed in~\cite{PhysRevE}, as follows.
For $\alpha\in(n-1,n)$, let $u(x)$ be $(n-1)$-times continuously differentiable on $(a,b)$ with its $n$-times derivative integrable on any subinterval of $[a,b]$, $ \lambda\geq 0.$ Then the {\em left} Riemann-Liouville {\em tempered} fractional derivative of order $\alpha$ is defined as
  \begin{equation}
  _{a}D_x^{\alpha,\lambda}u(x)=(e^{-\lambda x}\ _{a}D_x^{\alpha}e^{\lambda x})u(x)=\frac{e^{-\lambda x}}{\Gamma(n-\alpha)}\frac{d^n}{dx^n}\int_a^x \frac{e ^{\lambda\xi}u(\xi)}{(x-\xi)^{\alpha-n+1}}d\xi;
  \end{equation}
 the {\em right} Riemann-Liouville {\em tempered} fractional derivative of order $\alpha$ is defined as
  \begin{equation}
  _xD_b^{\alpha,\lambda}u(x)=(e^{\lambda x}\ _{x}D_b^{\alpha}e^{-\lambda x})u(x)=\frac{(-1)^ne^{\lambda x}}{\Gamma(n-\alpha)}\frac{d^n}{dx^n}\int_x^b \frac{e ^{-\lambda\xi}u(\xi)}{(\xi-x)^{\alpha-n+1}}d\xi.
  \end{equation}
  It can be seen that the Riemann-Liouville tempered fractional derivatives $_{a}D_x^{\alpha,\lambda}u(x)$ and $_xD_b^{\alpha,\lambda}u(x)$ can be transformed into the Riemann-Liouville tempered fractional derivatives $_{a}D_x^{\alpha}u(x)$ and $_xD_b^{\alpha}u(x)$ when $\lambda=0$. Variants of the left and right Riemann-Liouville tempered fractional derivatives are defined in \cite{PhysRevE,BAEUMER20102438} as follows
  \begin{equation}
_{a}\bm{D}_x^{\alpha,\lambda}u(x)=\begin{cases}
_{a}D_x^{\alpha,\lambda}u(x)-\lambda^\alpha u(x),&0<\alpha<1\\[2mm]
_{a}D_x^{\alpha,\lambda}u(x)-\alpha\lambda^{\alpha-1}\partial_x u(x)-\lambda^\alpha u(x),&1<\alpha<2;
\end{cases}
\label{a1}
\end{equation}
  and
  \begin{equation}
_x\bm{D}_b^{\alpha,\lambda}u(x)=\begin{cases}
_xD_b^{\alpha,\lambda}u(x)-\lambda^\alpha u(x),&0<\alpha<1\\[2mm]
_xD_b^{\alpha,\lambda}u(x)-\alpha\lambda^{\alpha-1}\partial_x u(x)-\lambda^\alpha u(x),&1<\alpha<2.
\end{cases}
\label{a2}
\end{equation}
In the above definitions, '$a$' and '$b$' can be extended to $-\infty$ and $\infty$, respectively.

\par There are quite a few numerical approximations for tempered fractional differential equations available in the literature already. Existing numerical techniques include finite differences (\cite{Chen2015High,Deng2017Fast,High2016Li,Zhang2016The}), finite elements (\cite{Deng2016Variational}) and also spectral methods (\cite{Huang2018Spectral,Zhao2016Spectral}).
Cartea and Del-Castillo-Negrete \cite{Cartea2006Fractional} presented a finite difference scheme for the tempered fractional Black-Scholes equation.
Li and Deng \cite{High2016Li} constructed higher-order discretizations, by weighted and shifted Gr\"unwald type approximations to tempered fractional derivatives. They also presented the stability and convergence properties of the second-order scheme for the tempered fractional diffusion equation. Zhang etc. \cite{Zhang2016The} provided a second-order discretization for the tempered fractional Black-Scholes equation and also gave stability and convergence results.

\par Based on the methods in \cite{ High2016Li}, we construct the third-order accurate scheme for the fractional diffusion equation, which can be proved to be unconditional stable for a range of $\alpha$-values and stability can be numerically confirmed when parameter $\alpha$ lies outside this interval. Besides, we also propose a third-order accurate scheme for the tempered fractional Black-Scholes equations.

\par This paper is organized as follows. In Section \ref{sec:Def}, we give discretization details for the tempered fractional derivatives. In Section \ref{sec:scheme}, we present the third-order schemes for the fractional diffusion equation and carry out the corresponding stability and error analysis. Section \ref{sec:example} contains some numerical results to confirm the accuracy and efficiency of the proposed discretization methods. We also consider the tempered fractional Black-Scholes equation and provide numerical experiments to confirm the scheme's accuracy. Finally, a short summary is made in the last section.

\section{Discretization of the tempered fractional derivatives}\label{sec:Def}
In this section, we introduce approximation accuracy of the shifted Gr\"unwald type difference operator for the Riemann-Liouville tempered fractional derivatives ~(\ref{a1}) and (\ref{a2}). We denote
$_{a}\mathfrak{D}_x^{\alpha,\lambda}u(x)=\ _{a}D_x^{\alpha,\lambda} u(x)-\lambda^\alpha u(x)$ and $_{x}\mathfrak{D}_b^{\alpha,\lambda}u(x)=\ _{x}D_b^{\alpha,\lambda} u(x)-\lambda^\alpha u(x),$ where '$a$' and '$b$' can be extended to $-\infty$ and $\infty$, respectively.
\begin{lemma} From~\cite{High2016Li}.
Let $u(x) \in L^1(\Omega)$, $_{-\infty}D_x^{\alpha,\lambda} u$ and its Fourier transform belong to $L^1({\Omega})$; $p\in\mathbb{R},\ h>0,\ \lambda\geq0$ and $n-1<\alpha\leq n$. Define the shifted Gr\"unwald type difference operator, as
\begin{equation}
A_{h,p}^{\alpha,\lambda}u(x):
=\frac{1}{h^\alpha}\sum_{k=0}^\infty\omega_k^{(\alpha)}e^{-(k-p)h\lambda}u(x-(k-p)h)
-\frac{1}{h^\alpha}\left(e^{ph\lambda}(1-e^{-h\lambda})^\alpha\right)u(x),
\end{equation}
then
\begin{equation}
A_{h,p}^{\alpha,\lambda}u(x)=_{-\infty}\mathfrak{D}_x^{\alpha,\lambda} u+O(h),
\end{equation}
where $\omega^{(\alpha)}_k=(-1)^k C_k^\alpha$, $k\ge0$ denote the normalized Gr\"unwald weights from~\cite{Podlubny}.
Let $ _{x}D_{+\infty}^{\alpha,\lambda} u$ and its Fourier transform belong to $L^1({\Omega})$, then we define
\begin{equation}
B_{h,p}^{\alpha,\lambda}u(x):=\frac{1}{h^\alpha}\sum_{k=0}^\infty\omega_k^{(\alpha)}e^{-(k-p)h\lambda}u(x+(k-p)h)-\frac{1}{h^\alpha}\left(e^{p h\lambda}(1-e^{-h\lambda})^\alpha\right)u(x).
\end{equation}
so that
\begin{equation}
B_{h,p}^{\alpha,\lambda}u(x)= _{x}\mathfrak{D}_\infty^{\alpha,\lambda}u(x)+O(h).
\end{equation}
\end{lemma}
To develop and analyze the higher-order discretization schemes, we repeat here a highly relevant theorem from~\cite{High2016Li}.

\begin{theorem} \label{T2-2} From~\cite{High2016Li}. Let $u(x) \in L^1(\mathbb{R})$, $_{-\infty}\mathfrak{D}_x^{\alpha+l,\lambda}u(x)$ and its Fourier transform belong to $L^1({\mathbb{R}})$. 
We define the left tempered-WSGD operator by
\begin{equation}
\label{eq:2-a}
_L\mathcal{D}_{h,p_1,p_2,...,p_m}^{\alpha,\gamma_1,\gamma_2,...,\gamma_m}
=\sum^{m}_{j=1}\gamma_j A_{h,p_j}^{\alpha,\lambda}u(x),
\end{equation}
where $p_j,\ \gamma_j\in\mathbb{R}$ (and they are determined by the Equations \eqref{eq:1-a}-\eqref{eq:1-d} to follow). Then, for any integer $m\geq l$, there exists an operator such that
\begin{equation}
_L\mathcal{D}_{h,p_1,p_2,...,p_m}^{\alpha,\gamma_1,\gamma_2,...,\gamma_m}=_{-\infty}\mathfrak{D}_x^{\alpha,\lambda}u(x)+O(h^l),
\end{equation}
uniformly for $x\in\mathbb{R}$.

Let $u(x) \in L^1(\mathbb{R})$, $_{x}\mathfrak{D}_b^{\alpha,\lambda}u(x)$ and its Fourier transform belong to $L^1(\mathbb{R})$; and define the right tempered-WSGD operator by
\begin{equation}
\label{eq:2-b}
_R\mathcal{D}_{h,p_1,p_2,...,p_m}^{\alpha,\gamma_1,\gamma_2,...,\gamma_m}
=\sum^{m}_{j=1}\gamma_j B_{h,p_j}^{\alpha,\lambda}u(x),
\end{equation}
Then, for any integer $m\geq l$, there exists an operator such that
\begin{equation}
_R\mathcal{D}_{h,p_1,p_2,...,p_m}^{\alpha,\gamma_1,\gamma_2,...,\gamma_m}=_{x}\mathfrak{D}_\infty^{\alpha,\lambda}u(x)+O(h^l),
\end{equation}
uniformly for $x\in\mathbb{R}$.

Let $p_j,$ $\gamma_j$ be real-valued.
For $l=2$, $p_j,$ $\gamma_j$ should satisfy the following conditions
\begin{equation}
\label{eq:1-a}
\left\{
\begin{aligned}
&\sum_{j=1}^m\gamma_j=1,\\
&\sum_{j=1}^m\gamma_j\left[p_j-\frac{\alpha}{2}\right]=0.
\end{aligned}
\right.
\end{equation}

For $l=3$, $p_j,$ $\gamma_j$ should satisfy the following conditions
\begin{equation}
\label{eq:1-b}
\left\{
\begin{aligned}
&\sum_{j=1}^m\gamma_j=1,\\
&\sum_{j=1}^m\gamma_j\left[p_j-\frac{\alpha}{2}\right]=0,\\
&\sum_{j=1}^m\gamma_j\left[\frac{p_j^2}{2}-\frac{\alpha p_j}{2}+\frac{\alpha}{6}+\frac{\alpha(\alpha-1)}{8}\right]=0.
\end{aligned}
\right.
\end{equation}

For $l=4$, $p_j,$ $\gamma_j$ should satisfy
\begin{equation}
\label{eq:1-c}
\left\{
\begin{aligned}
&\sum_{j=1}^m\gamma_j=1,\\
&\sum_{j=1}^m\gamma_j\left[p_j-\frac{\alpha}{2}\right]=0,\\
&\sum_{j=1}^m\gamma_j\left[\frac{p_j^2}{2}-\frac{\alpha p_j}{2}+\frac{\alpha}{6}+\frac{\alpha(\alpha-1)}{8}\right]=0,\\
&\sum_{j=1}^m\gamma_j\left[\frac{p_j^3}{6}-\frac{\alpha p_j^2}{4}+\left(\frac{\alpha}{6}+\frac{\alpha(\alpha-1)}{8}\right)p_j-\frac{\alpha}{24}-\frac{\alpha(\alpha-1)}{12}-\frac{\alpha(\alpha-1)(\alpha-2)}{48}\right]=0,
\end{aligned}
\right.
\end{equation}

while for $l=5$, $p_j,$ $\gamma_j$ should satisfy
\begin{equation}
\label{eq:1-d}
\left\{
\begin{aligned}
&\sum_{j=1}^m\gamma_j=1,\\
&\sum_{j=1}^m\gamma_j\left[p_j-\frac{\alpha}{2}\right]=0,\\
&\sum_{j=1}^m\gamma_j\left[\frac{p_j^2}{2}-\frac{\alpha p_j}{2}+\frac{\alpha}{6}+\frac{\alpha(\alpha-1)}{8}\right]=0,\\
&\sum_{j=1}^m\gamma_j\left[\frac{p_j^3}{6}-\frac{\alpha p_j^2}{4}+\left(\frac{\alpha}{6}+\frac{\alpha(\alpha-1)}{8}\right)p_j
-\frac{\alpha}{24}-\frac{\alpha(\alpha-1)}{12}-\frac{\alpha(\alpha-1)(\alpha-2)}{48}\right]=0,\\
&\sum_{j=1}^m\gamma_j\left[\frac{p_j^4}{24}-\frac{\alpha p_j^3}{4}+\frac{1}{2}\left(\frac{\alpha}{6}+\frac{\alpha(\alpha-1)}{8}\right)p^2_j
-\left(\frac{\alpha}{24}-\frac{\alpha(\alpha-1)}{12}-\frac{\alpha(\alpha-1)(\alpha-2)}{48}\right)p_j\right.\\
&\left.+\frac{\alpha}{120}+\frac{5\alpha(\alpha-1)}{144}+\frac{\alpha(\alpha-1)(\alpha-2)}{48}+\frac{\alpha(\alpha-1)(\alpha-2)(\alpha-3)}{384}\right]=0.
\end{aligned}
\right.
\end{equation}
\end{theorem}
Considering a well-defined function $u(x)$ on the bounded interval $[a,b]$, which can be zero for $x<a$ or $x>b$, its $\alpha$-th order left and right Riemann-Liouville tempered fractional derivatives at point $x$ can be approximated by the following tempered-WSGD operators
\begin{eqnarray}
_{a}\mathfrak{D}_x^{\alpha,\lambda}u(x)&=&\sum_{j=1}^{m}\frac{\gamma_j}{h^\alpha}\Bigg(\sum^{\left[\frac{x-a}{h}\right]+p_j}_{k=0}\omega_{k}^{(\alpha)}u(x-(k-p_j)h)
-e^{p_j h\lambda}(1-e^{-h\lambda})^\alpha u(x)\Bigg)\nonumber\\
&\quad&+O(h^l),\nonumber\\
_{x}\mathfrak{D}_b^{\alpha,\lambda}u(x)&=&\sum_{j=1}^{m}\frac{\gamma_j}{h^\alpha}\Bigg(\sum^{\left[\frac{b-x}{h}\right]+p_j}_{k=0}\omega_{k}^{(\alpha)}u(x+(k-p_j)h)
-e^{p_j h\lambda}(1-e^{-h\lambda})^\alpha u(x)\Bigg)\nonumber\\
&\quad&+O(h^l),
\end{eqnarray}

Take $p_1=1,\ p_2=0,$ and $p_3=-1$ for the second-order scheme, with~(\ref{eq:1-a}). The parameters $\gamma_j~(j=1,2,3)$ satisfy the linear system as follows,
\begin{equation}
\left\{
\begin{aligned}
&\gamma_1=\frac{\alpha}{2}+\gamma_3,\\
&\gamma_2=\frac{2-\alpha}{2}-2\gamma_3.
\end{aligned}
\right.
\end{equation}
The second-order operators are then given by
\begin{equation}\label{2-22-1}
\begin{aligned}
_L\mathcal{D}_{h,1,0,-1}^{\alpha,\gamma_1,\gamma_2,\gamma_3}u(x)
=\sum_{j=1}^{3}\frac{\gamma_j}{h^\alpha}\Bigg(\sum^{\left[\frac{x-a}{h}\right]+p_j}_{k=0}\omega_{k}^{(\alpha)}u(x-(k-p_j)h)
-e^{p_j h\lambda}(1-e^{-h\lambda})^\alpha u(x)\Bigg)&\\
=\frac{1}{h^\alpha}\Bigg(\sum^{\left[\frac{x-a}{h}\right]+1}_{k=0}g_{k,\lambda}^{(2,\alpha)}u(x-(k-1)h)
-\left(\gamma_1e^{h\lambda}+\gamma_2+\gamma_3e^{-h\lambda}\right)(1-e^{-h\lambda})^\alpha u(x)\Bigg)&,
\end{aligned}
\end{equation}
and
\begin{equation}\label{2-23-1}
\begin{aligned}
_R\mathcal{D}_{h,1,0,-1}^{\alpha,\gamma_1,\gamma_2,\gamma_3}u(x_j)
=\sum_{j=1}^{3}\frac{\gamma_j}{h^\alpha}\Bigg(\sum^{\left[\frac{b-x}{h}\right]+p_j}_{k=0}\omega_{k}^{(\alpha)}u(x+(k-p_j)h)
-e^{p_j h\lambda}(1-e^{-h\lambda})^\alpha u(x)\Bigg)&\\
=\frac{1}{h^\alpha}\Bigg(\sum^{\left[\frac{b-x}{h}\right]+1}_{k=0}g_{k,\lambda}^{(2,\alpha)}u(x+(k-1)h)
-\left(\gamma_1e^{h\lambda}+\gamma_2+\gamma_3e^{-h\lambda}\right)(1-e^{-h\lambda})^\alpha u(x)\Bigg)&.
\end{aligned}
\end{equation}
The weights are given by
\begin{equation}\label{eq:2-24}
\begin{aligned}
&g_{0,\lambda}^{(2,\alpha)}=\gamma_1\omega_0^{(\alpha)}e^{h\lambda},\
g_{1,\lambda}^{(2,\alpha)}=\gamma_1\omega_1^{(\alpha)}+\gamma_2\omega_0^{(\alpha)},\\
&g_{k,\lambda}^{(2,\alpha)}=\left(\gamma_1\omega_{k}^{(\alpha)}+\gamma_2\omega_{k-1}^{(\alpha)}+\gamma_3\omega_{k-2}^{(\alpha)}\right)e^{-(k-1)h\lambda},\ k\geq2.
\end{aligned}
\end{equation}

The third-order scheme is obtained by the following result with $\gamma_j~(j=1,2,\cdots,m)$, satisfying Equation~(\ref{eq:1-b}) in Theorem \ref{T2-2}. Let $m=4,\ p_1=1,\ p_2=0,\ p_3=-1,\ p_4=-2.$ We have,
\begin{equation}
\left\{
\begin{aligned}
&\gamma_1=\frac{\alpha^2}{8}+\frac{5}{24}\alpha-\gamma_4,\\
&\gamma_2=-\frac{\alpha^2}{4}+\frac{1}{12}\alpha+1+3\gamma_4,\\
&\gamma_3=\frac{\alpha^2}{8}-\frac{7}{24}\alpha-3\gamma_4.
\end{aligned}
\right.
\end{equation}
The third-order operators are then given by
\begin{equation}\label{2-26-1}
\begin{aligned}
&_L\mathcal{D}_{h,1,0,-1,-2}^{\alpha,\gamma_1,\gamma_2,\gamma_3,\gamma_4}u(x)
=\sum_{j=1}^{4}\frac{\gamma_j}{h^\alpha}\Bigg(\sum^{\left[\frac{x-a}{h}\right]+p_j}_{k=0}\omega_{k}^{(\alpha)}u(x-(k-p_j)h)
-e^{p_j h\lambda}(1-e^{-h\lambda})^\alpha u(x)\Bigg)\\
&\qquad=\frac{1}{h^\alpha}\left(\sum^{\left[\frac{x-a}{h}\right]+1}_{k=0}g_{k,\lambda}^{(3,\alpha)}u(x-(k-1)h)
-\left(\gamma_1e^{h\lambda}+\gamma_2+\gamma_3e^{-h\lambda}+\gamma_4e^{-2h\lambda}\right)(1-e^{-h\lambda})^\alpha u(x)\right),
\end{aligned}
\end{equation}
and
\begin{equation}\label{2-27-1}
\begin{aligned}
&_R\mathcal{D}_{h,1,0,-1,-2}^{\alpha,\gamma_1,\gamma_2,\gamma_3,\gamma_4}u(x)
=\sum_{j=1}^{4}\frac{\gamma_j}{h^\alpha}\Bigg(\sum^{\left[\frac{b-x}{h}\right]+p_j}_{k=0}\omega_{k}^{(\alpha)}u(x-(k-p_j)h)
-e^{p_j h\lambda}(1-e^{-h\lambda})^\alpha u(x)\Bigg)\\
&\qquad=\frac{1}{h^\alpha}\left(\sum^{\left[\frac{b-x}{h}\right]+1}_{k=0}g_{k,\lambda}^{(3,\alpha)}u(x+(k-1)h)
-\left(\gamma_1e^{h\lambda}+\gamma_2+\gamma_3e^{-h\lambda}+\gamma_4e^{-2h\lambda}\right)(1-e^{-h\lambda})^\alpha u(x)\right).
\end{aligned}
\end{equation}
The weights are found to be
\begin{equation}\label{eq:2-28}
\begin{aligned}
&g_{0,\lambda}^{(3,\alpha)}=\gamma_1\omega_0^{(\alpha)}e^{h\lambda},\
g_{1,\lambda}^{(3,\alpha)}=\gamma_1\omega_1^{(\alpha)}+\gamma_2\omega_0^{(\alpha)},\\
&g_{2,\lambda}^{(3,\alpha)}=\left(\gamma_1\omega_2^{(\alpha)}+\gamma_2\omega_1^{(\alpha)}+\gamma_3\omega_0^{(\alpha)}\right)e^{-h\lambda},\\
&g_{k,\lambda}^{(3,\alpha)}=\left(\gamma_1\omega_{k}^{(\alpha)}+\gamma_2\omega_{k-1}^{(\alpha)}
+\gamma_3\omega_{k-2}^{(\alpha)}+\gamma_4\omega_{k-3}^{(\alpha)}\right)e^{-(k-1)h\lambda},\ k\geq3.
\end{aligned}
\end{equation}
We will discretize and analyze this latter scheme in the section to follow.

\section{Numerical schemes for the tempered fractional diffusion equation}\label{sec:scheme}
\par Here we construct a high-order scheme based on the tempered-WSGD operators in space and the Crank--Nicolson scheme in time for the tempered fractional diffusion equation. Then we establish the stability and convergence for the third-order scheme.
\par Consider the following tempered fractional diffusion equation, from~\cite{High2016Li},
\begin{equation}
\label{4-0}
\left\{
\begin{aligned}
&\frac{\partial u(x,t)}{\partial t}
=c_l \cdot \left(_{a}\mathfrak{D}_x^{\alpha,\lambda}u(x,t)\right)+c_r \cdot \left(_x\mathfrak{D}_{b}^{\alpha,\lambda}u(x,t)\right)
+f(x,t),\\
&(x,t)\in(a,b)\times(0,T)\\
&u(a,t)=\Phi_l(t),u(b,t)=\Phi_r(t),\ t\in(0,T)\\
&u(x,T)=S(x),\ x\in(a,b),
\end{aligned}
\right.
\end{equation}
where $\alpha\in(1,2)$, $c_l$ and $c_r$ are constants in front of the fractional derivatives with $c_l+c_r\neq0$ which usually control the bias of the diffusion. And if $c_l\neq0$, then $\Phi_l(t)\equiv0$; if $c_r\neq0$, then $\Phi_r(t)\equiv0$.
\par Let $t_j=j\tau$ $(0\leq t_j\leq T,\ j=0,\dots,N)$ and $x_i=a+ih$ $(a\leq x_i\leq b,\ i=0,\dots,M)$, where $\tau=T/N$ and $h=(b-a)/M$. Using the k-th order tempered-WSGD operators $_L\mathcal{D}^{\alpha,\lambda_1}_{h,k}~=~_L\mathcal{D}_{h,p_1,p_2,...,p_m}^{\alpha,\gamma_1,\gamma_2,...,\gamma_m}$ and $_R\mathcal{D}^{\alpha,\lambda_2}_{h,k}~=~_R\mathcal{D}_{h,p_1,p_2,...,p_m}^{\alpha,\gamma_1,\gamma_2,...,\gamma_m}$ for the tempered fractional derivatives and Crank-Nicolson time discretization, the numerical scheme for \eqref{4-0} reads
\begin{equation}
\label{4-2}
\begin{aligned}
&\frac{u^{j+1}_{i}-u_{i}^j}{\tau}
=c_l\left(_L\mathcal{D}^{\alpha,\lambda}_{h,k}{u_{i}^{j+\frac{1}{2}}}\right)
+c_r\left(_R\mathcal{D}^{\alpha,\lambda}_{h,k}{u_{i}^{j+\frac{1}{2}}}\right)
+f^{j+\frac{1}{2}}_{i}+O(\tau^2+ h^k),
\end{aligned}
\end{equation}
where $u_{i}^j$ represents the solution of \eqref{4-0} at the point $(x_i,t_j),$ and $f_i^{j+\frac{1}{2}}=f(x_i,t_{j+\frac{1}{2}}).$
\par Then, we obtain,
\begin{eqnarray}
\label{4-4}
u_{i}^{j+1} &-&\frac{\tau}{2}\left[
c_l\left(_L\mathcal{D}^{\alpha,\lambda}_{h,k}u_{i}^{j+1}\right)
+c_r\left(_R\mathcal{D}^{\alpha,\lambda}_{h,k}u_{i}^{j+1}\right)
\right]\nonumber \\
&=&u_{i}^j
+\frac{\tau}{2}\left[
c_l\left(_L\mathcal{D}^{\alpha,\lambda}_{h,k}u_{i}^{j}\right)
+c_r\left(_R\mathcal{D}^{\alpha,\lambda}_{h,k}u_{i}^{j}\right)
\right]+\tau f_i^{j+\frac{1}{2}}+O(\tau^3+\tau h^k).
\end{eqnarray}
\par We denote by $U_{i}^{j}$ the solution of the numerical scheme for \eqref{4-0} at point $(x_i,t_j)$ and $F_i^{j+\frac{1}{2}}=\frac{1}{2}(f_i^{j}+f_i^{j+1})$. The numerical scheme can now be written as

\begin{eqnarray}
\label{4-7}
U_{i}^{j+1}
&-&\frac{\tau}{2}\left[
c_l \cdot \left(_L\mathcal{D}^{\alpha,\lambda}_{h,k}U_{i}^{j+1}\right)
+c_r \cdot \left(_R\mathcal{D}^{\alpha,\lambda}_{h,k}U_{i}^{j+1}\right)
\right]\nonumber \\
&=&U_{i}^j
+\frac{\tau}{2}\left[
c_l \cdot \left(_L\mathcal{D}^{\alpha,\lambda}_{h,k}U_{i}^{j}\right)
+c_r \cdot \left(_R\mathcal{D}^{\alpha,\lambda}_{h,k}U_{i}^{j}\right)
\right]+\tau F_i^{j+\frac{1}{2}}.
\end{eqnarray}

\par Denote $U^n=(U_1^n,U^n_2,...,U^n_{M-1})^T$, $\phi^{(k,m)}(\lambda)=\sum_{j=1}^m\gamma_j e^{p_j h\lambda}(1-e^{-h\lambda})^\alpha$, and
\begin{equation}\label{e3-5}
B_{k,\lambda}=\left(
\begin{array}{ccccc}
g_{1,\lambda}^{(k,\alpha)}-\phi^{(k,m)}(\lambda)&g_{0,\lambda}^{(k,\alpha)}
& & &\\
g_{2,\lambda}^{(k,\alpha)}&g_{1,\lambda}^{(k,\alpha)}-\phi^{(k,m)}(\lambda)&g_{0,\lambda}^{(k,\alpha)}
& &\\
\vdots      &g_{2,\lambda}^{(k,\alpha)}&g_{1,\lambda}^{(k,\alpha)}-\phi^{(k,m)}(\lambda)&\ddots&\\
g_{n-2,\lambda}^{(k,\alpha)}&\dots&\ddots&\ddots &g_{0,\lambda}^{(k,\alpha)}
\\
g_{n-1,\lambda}^{(k,\alpha)}&g_{n-2,\lambda}^{(k,\alpha)}&\dots&g_{2,\lambda}^{(k,\alpha)}&g_{1,\lambda}^{(k,\alpha)}-\phi^{(k,m)}(\lambda)
\end{array}\right).
\end{equation}
\par The matrix form for \eqref{4-7} can be written as follows
\begin{equation}
\label{eq:4-c}
\begin{aligned}
&\left(I-\frac{\tau}{2h^\alpha}(c_lB_{k,\lambda}+c_rB_{k,\lambda}^T)\right)U^{j+1}
=\left(I+\frac{\tau}{2h^\alpha}(c_lB_{k,\lambda}+c_rB_{k,\lambda}^T)\right)U^{j}+\tau \hat{F}_i^{j+\frac{1}{2}},
\end{aligned}
\end{equation}
where
\begin{equation}
\hat{F}^{n+\frac{1}{2}}=\left(
\begin{array}{c}
F^{n+\frac{1}{2}}_1\\
F^{n+\frac{1}{2}}_2\\
\vdots\\
F^{n+\frac{1}{2}}_{M-2}\\
F^{n+\frac{1}{2}}_{M-1}
\end{array}\right)+\frac{1}{2h^\alpha}
\left(
\begin{array}{c}
c_lg_{2,\lambda}^{(k,\alpha)}+c_r g_{0,\lambda}^{(k,\alpha)}\\
c_lg_{3,\lambda}^{(k,\alpha)}\\
\vdots\\
c_lg_{M-1,\lambda}^{(k,\alpha)}\\
c_lg_{M,\lambda}^{(k,\alpha)}
\end{array}\right)(U^n_0+U_0^{n+1})
+\frac{1}{2h^\alpha}
\left(
\begin{array}{c}
c_r g_{M,\lambda}^{(k,\alpha)}\\
c_r g_{M-1,\lambda}^{(k,\alpha)}\\
\vdots\\
c_r g_{3,\lambda}^{(k,\alpha)}\\
c_r g_{2,\lambda}^{(k,\alpha)}+c_lg_{0,\lambda}^{(k,\alpha)}
\end{array}\right)
(U^n_M+U_M^{n+1}).
\end{equation}
\subsection{The second-order numerical scheme}
\par In this subsection, we detail the second-order scheme for the tempered fractional diffusion equation.  
Using the tempered-WSGD operators, $_L\mathcal{D}^{\alpha,\lambda_1}_{h,2}=_L\mathcal{D}_{h,-1,0,1}^{\alpha,\gamma_1,\gamma_2,\gamma_3}$ and $_R\mathcal{D}^{\alpha,\lambda_2}_{h,2}=_R\mathcal{D}_{h,-1,0,1}^{\alpha,\gamma_1,\gamma_2,\gamma_3}$, for the tempered fractional derivatives and the Crank-Nicolson discretization in time, we find the following second-order scheme for \eqref{4-0},
\begin{eqnarray}
\label{4-1-1}
u_{i}^{j+1} & - & \frac{\tau}{2}\left[
c_l\left(_L\mathcal{D}^{\alpha,\lambda}_{h,2}u_{i}^{j+1}\right)
+c_r\left(_R\mathcal{D}^{\alpha,\lambda}_{h,2}u_{i}^{j+1}\right)
\right] \nonumber \\
&=&u_{i}^j
+\frac{\tau}{2}\left[
c_l\left(_L\mathcal{D}^{\alpha,\lambda}_{h,2}u_{i}^{j}\right)
+c_r\left(_R\mathcal{D}^{\alpha,\lambda}_{h,2}u_{i}^{j}\right)
\right]+\tau f_i^{j+\frac{1}{2}}+O(\tau^3+\tau h^2).
\end{eqnarray}
The numerical scheme can be written as,
\begin{eqnarray}
\label{4-1-2}
U_{i}^{j+1}
&-&\frac{\tau}{2}\left[
c_l\left(_L\mathcal{D}^{\alpha,\lambda}_{h,2}U_{i}^{j+1}\right)
+c_r\left(_R\mathcal{D}^{\alpha,\lambda}_{h,2}U_{i}^{j+1}\right)
\right]\nonumber\\
&=&U_{i}^j
+\frac{\tau}{2}\left[
c_l\left(_L\mathcal{D}^{\alpha,\lambda}_{h,2}U_{i}^{j}\right)
+c_r\left(_R\mathcal{D}^{\alpha,\lambda}_{h,2}U_{i}^{j}\right)
\right]+\tau F_i^{j+\frac{1}{2}},
\end{eqnarray}
and the following matrix form results,
\begin{equation}
\label{eq:4-1-c}
\left(I-\frac{\tau}{2h^\alpha}(c_l B_{\lambda}+c_r B_{\lambda}^T)\right)U^{j+1}
=\left(I+\frac{\tau}{2h^\alpha}(c_l B_{\lambda}+c_r B_{\lambda}^T)\right)U^{j}+\tau \hat{F}_i^{j+\frac{1}{2}},
\end{equation}
where 
$B_{\lambda}=B_{2,\lambda}$, as defined in \eqref{e3-5}.

\par The stability and convergence for the second-order scheme have already been presented in \cite{High2016Li} in the following theorem, based on the lemma below.
\begin{lemma} From~\cite{High2016Li}.\label{l3-1}
For $1<\alpha<2$ and $\lambda\geq0$, if $$max\{\frac{(2-\alpha)(\alpha^2+\alpha-8)}{2(\alpha^2+3\alpha+2)},\frac{(1-\alpha)(\alpha^2+2\alpha)}{2(\alpha^2+3\alpha+4)}\}
<\gamma_3<\frac{(2-\alpha)(\alpha^2+2\alpha-3)}{2(\alpha^2+3\alpha+2)},$$ then the weighs coefficients $\omega^{(\alpha)}_k$ and $g_{k,\lambda}^{(2,\alpha)}$ satisfy
\begin{enumerate}
  \item $\omega^{(\alpha)}_0=1$, $\omega^{(\alpha)}_1=-\alpha,$ $0\leq...\leq\omega^{(\alpha)}_3\leq\omega^{(\alpha)}_2\leq1$, $\sum_{k=0}^{\infty}\omega^{(\alpha)}_k=0$.
  \item $g_{1,\lambda}^{(2,\alpha)}\leq0$, $g_{0,\lambda}^{(2,\alpha)}+g_{2,\lambda}^{(2,\alpha)}\geq0$, $g_{k,\lambda}^{(2,\alpha)}\geq0(k\geq3).$
\end{enumerate}
\end{lemma}

\begin{theorem} From~\cite{High2016Li}.
For $1<\alpha<2$, and $\lambda_1,\ \lambda_2\geq0$, if $a_1<\gamma_3<a_2$,  the numerical scheme \eqref{eq:4-c} is stable for $$a_1=max\{\frac{(2-\alpha)(\alpha^2+\alpha-8)}{2(\alpha^2+3\alpha+2)},\frac{(1-\alpha)(\alpha^2+2\alpha)}{2(\alpha^2+3\alpha+4)}\}$$ and
$$a_2=\frac{(2-\alpha)(\alpha^2+2\alpha-3)}{2(\alpha^2+3\alpha+2)}.$$ Denoting $e_i^j=u^j_i-U^j_i,\; i=1,2,...,M-1$ and $E^j=(e^j_1,e^j_2,...,e^j_{M-1})^T, j=1,2,...,N$, moreover, it is found that
\begin{equation}
\norm{E^j}_h \leq c(\tau^2+h^2),\ 1\leq j\leq N-1.
\end{equation}
\end{theorem}
\subsection{The third-order numerical scheme}
In this subsection, we consider third-order accurate scheme for the tempered fractional diffusion equation.
Denote $$g_{k,\lambda}=g^{(3,\alpha)}_{k,\lambda},\ \phi(\lambda)=\phi^{(3,4)}(\lambda)=\left(\gamma_1e^{h\lambda}+\gamma_2+\gamma_3e^{-h\lambda}+\gamma_4e^{-2h\lambda}\right)(1-e^{-h\lambda})^\alpha.$$
Using the tempered-WSGD operators, $_L\mathcal{D}^{\alpha,\lambda}_{h,3}=_L\mathcal{D}_{h,-1,0,1,2}^{\alpha,\gamma_1,\gamma_2,...,\gamma_4}$ and $_R\mathcal{D}^{\alpha,\lambda}_{h,3}=_R\mathcal{D}_{h,-1,0,1,2}^{\alpha,\gamma_1,\gamma_2,...,\gamma_4}$, for the tempered fractional derivatives and the Crank-Nicolson time discretization, we find the following discretization for \eqref{4-0},
\begin{eqnarray}
\label{4-2-1}
u_{i}^{j+1} &-&\frac{\tau}{2}\left[
c_l\left(_L\mathcal{D}^{\alpha,\lambda}_{h,3}u_{i}^{j+1}\right)
+c_r\left(_R\mathcal{D}^{\alpha,\lambda}_{h,3}u_{i}^{j+1}\right)
\right] \nonumber \\
&=&u_{i}^j
+\frac{\tau}{2}\left[
c_l\left(_L\mathcal{D}^{\alpha,\lambda}_{h,3}u_{i}^{j}\right)
+c_r\left(_R\mathcal{D}^{\alpha,\lambda}_{h,3}u_{i}^{j}\right)
\right]+\tau f_i^{j+\frac{1}{2}}+O(\tau^3+\tau h^3).
\end{eqnarray}
This numerical scheme is then written as,
\begin{eqnarray}
\label{4-2-2}
U_{i}^{j+1}
&-&\frac{\tau}{2}\left[
c_l\left(_L\mathcal{D}^{\alpha,\lambda}_{h,3}U_{i}^{j+1}\right)
+c_r\left(_R\mathcal{D}^{\alpha,\lambda}_{h,3}U_{i}^{j+1}\right)
\right]\nonumber\\
&=&U_{i}^j
+\frac{\tau}{2}\left[
c_l\left(_L\mathcal{D}^{\alpha,\lambda}_{h,3}U_{i}^{j}\right)
+c_r\left(_R\mathcal{D}^{\alpha,\lambda}_{h,3}U_{i}^{j}\right)
\right]+\tau F_i^{j+\frac{1}{2}}.
\end{eqnarray}
The matrix form looks as follows,
\begin{equation}
\label{eq:4-2-c}
\begin{aligned}
&\left(I-\frac{\tau}{2h^\alpha}(c_lB_{\lambda}+c_rB_{\lambda}^T)\right)U^{j+1}
=\left(I+\frac{\tau}{2h^\alpha}(c_lB_{\lambda}+c_rB_{\lambda}^T)\right)U^{j}+\tau \hat{F}_i^{j+\frac{1}{2}},
\end{aligned}
\end{equation}
where 
again, $B_{\lambda}=B_{3,\lambda}$, as defined in \eqref{e3-5}.

We will now analyze the stability and convergence of this third-order scheme. 
First of all, we introduce the following lemmas.
\begin{lemma} See, for example,~\cite{QuarteroniNumerical}.\label{l3-2}
A real-valued matrix $A $ of order $n$ is positive definite if and only if its symmetric part, $H=\frac{A+A^T}{2}$, is positive definite; $H$ is positive definite if and only if the eigenvalues of $H$ are positive.
For any eigenvalue $\mu$ of $A$, we have
\begin{equation}
\mu_{min}(H)\leq Re(\mu(A))\leq \mu_{max}(H),
\end{equation}
where $Re(\mu(A))$ represents the real part of $\mu$, and $\mu_{min}(H)$ and $\mu_{max}(H)$ are the minimum and maximum of the eigenvalues of $H$, respectively.
\end{lemma}
To obtain the stability, we introduce the definition of the  Toeplitz matrix $T_n$ and its generating function $f$.
\begin{definition} See, for example, \cite{CHAN1991Toeplitz}
 Let the Toeplitz matrix $T_n$ be of the following form,
 \begin{equation}
T_n=\left(
\begin{array}{ccccc}
t_0&t_{-1}&\cdots&t_{2-n}&t_{1-n}\\
t_1&t_0&t_{-1}&\cdots&t_{2-n}\\
\vdots&t_1&t_0&\ddots&\vdots\\
t_{n-2}&\cdots&\ddots&\ddots&t_{-1}\\
t_{n-1}&t_{n-2}&\ddots&t_1&t_0
\end{array}\right).
\end{equation}
If the diagonals $\{t_k\}^{n-1}_{ k=-n+1}$ are the Fourier coefficients of a function $f$, i.e.,
\begin{equation}
t_k=\frac{1}{2\pi}\int_{-\pi}^\pi f(x)e^{-ikx}dx,
\end{equation}
then function $f$ is called the generating function of $T_n$.
\end{definition}
\begin{lemma}(Grenander-Szeg\"o Theorem \cite{CHAN1991Toeplitz})\label{l3-6}
For a Toeplitz matrix $T_n$, we denote by $\mu_{min}(T_n)$ and $\mu_{max}(T_n)$ the smallest and largest eigenvalues of $T_n$, respectively. If $f$ is a $2\pi$-periodic continuous real-valued function, defined on $[-\pi,\pi]$, then $$f_{min} \leq \mu_{min}(T_n) \leq \mu_{max}(T_n) \leq f_{max},$$ where $f_{min}$ and $f_{max}$ denote the minimum and maximum values of $f(x)$. Moreover, if $f_{min} < f_{max}$, then all the eigenvalues of $T_n$ satisfy $$f_{min}<\mu(T_n) < f_{max},$$ for all $n > 0$; and furthermore if $f_{min}\geq0$, then $T_n$ is positive definite.

\end{lemma}
Next, we define the functions,
\begin{equation}
\label{eq:f}
f_B(x)=\sum_{k=0}^{N-1}g_{k,\lambda}e^{i(k-1)x}-\phi(\lambda),\ f_{B^T}(x)=\sum_{k=0}^{N-1}g_{k,\lambda}e^{-i(k-1)x}-\phi(\lambda),
\end{equation}
and
\begin{equation}\label{eq:fh}
f(\alpha,\lambda;x)=\frac{f_B(x)+f_{B^T}(x)}{2}.
\end{equation}
It's straightforward to get the following lemma by Lemmas \ref{l3-2} - \ref{l3-6}.

\begin{lemma}\label{l3-7}
Let the matrices $B_{\lambda}$ and $B_{\lambda}^T$ be given via the numerical scheme \eqref{e3-5}. For $\lambda\geq0,\ h>0$ and $\alpha\in[1,2]$, if we can find (analytically, or with the help of numerical techniques) values of $\gamma_j$ for which the generating functions $f(\alpha,\lambda;x)$ of $B_{\lambda}$ are negative, then the eigenvalues of the matrix $B_{\lambda}$ are negative, too.
\end{lemma}
\begin{corollary}\label{T3-8}
If, for $1<\alpha<2$, the generating functions $f(\alpha,\lambda;x)$, given in \eqref{eq:fh}, are negative, the numerical scheme \eqref{4-2-2} is stable.
\end{corollary}
\begin{proof}
Let $\mathcal{M}=\frac{\tau}{2h^\alpha}(c_lB_{\lambda}+c_rB_{\lambda}^T)$. We find,
\begin{equation}
\begin{aligned}
\frac{\mathcal{M}+\mathcal{M}^T}{2}
=&\frac{\tau}{4h^\alpha}(c_l(B_{\lambda}+B^T_{\lambda})+c_r(B_{\lambda}+B_{\lambda}^T))
\end{aligned}
\end{equation}

With $\mu(\mathcal{M})$ an eigenvalue of matrix $\mathcal{M}$, it follows that $\mu(\mathcal{M})<0$ when $f(\alpha,\lambda;x)<0$ by Lemmas \ref{l3-2}-\ref{l3-6}. Then $\frac{1+\mu(\mathcal{M})}{1-\mu(\mathcal{M})}<1$ is an eigenvalue of matrix $|\mathrm{I}-\mathcal{M}|^{-1}|\mathrm{I}+\mathcal{M}|$. Hence, the numerical scheme \eqref{4-2-2} is stable.
\end{proof}

For $\alpha\in(1,2)$, we denote $$a_3=max\{\frac{\frac{\alpha^5}{8}+\frac{7}{12}\alpha^4-\frac{5}{8}\alpha^3-\frac{49}{12}\alpha^2+3\alpha}{\alpha^3+6\alpha^2+11\alpha+6},\ \frac{\frac{\alpha^5}{8} + \frac{\alpha^4}{3} - \frac{ 67}{24}\alpha^3 -\frac{ 23}{6}\alpha^2 +\frac{175}{6}\alpha - 30}{\alpha^3+6\alpha^2+11\alpha+6}\},$$
and $$a_4=min\{ \frac{\frac{1}{8}\alpha^4+\frac{7}{12}\alpha^3+\frac{1}{8}\alpha^2-\frac{13}{6}\alpha}{\alpha^2+5\alpha+8},\ \frac{\frac{\alpha^5}{8}+\frac{11}{24}\alpha^4-\frac{41}{24}\alpha^3 -\frac{107}{24}\alpha^2+\frac{163}{12}\alpha-8}{\alpha^3+6\alpha^2+11\alpha+6} \}.$$ Impact of varying $a_3$ and $a_4$ is illustrated in Figure \ref{fig:1}. It can be seen that when $\alpha \in (1.26,1.71)$, $a_3< a_4$ and $(a_3,a_4)\neq\emptyset$.
\begin{figure}[h]
\begin{center}
\includegraphics[width=11.5cm,height=8cm]{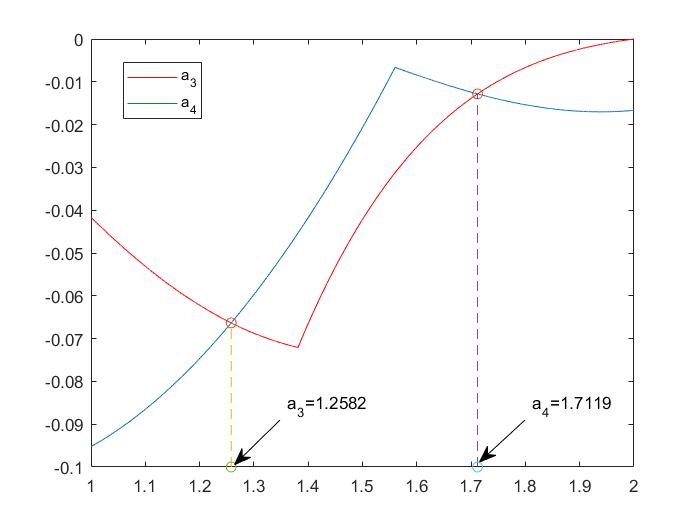}
\end{center}
\caption{$a_3$, $a_4$ with $\alpha\in(1,2)$}
\label{fig:1}
\end{figure}

For $\alpha\in(1.26,1.71)$, we obtain the following result, similar to Lemma \ref{l3-1}.
\begin{theorem}\label{l3-8}
For $\alpha\in(1.26,1.71)$, $\lambda\geq0$ and $a_3\leq\gamma_4\leq a_4$, then there exists
$g_{1,\lambda}\leq0,\ g_{0,\lambda}+g_{2,\lambda}\geq0,\ g_{k,\lambda}\geq0(k\geq3).$
\end{theorem}
\begin{proof}
For the term $g_{0,\lambda}+g_{2,\lambda}$, we have
\begin{equation}
\begin{aligned}
g_{0,\lambda}+g_{2,\lambda}
=&\gamma_1\omega_0^{(\alpha)}e^{h\lambda}+\left(\gamma_1\omega_2^{(\alpha)}+\gamma_2\omega_1^{(\alpha)}+\gamma_3\omega_0^{(\alpha)}\right)e^{-h\lambda}\\
=&\left(\frac{\alpha^2}{8}+\frac{5}{24}\alpha-\gamma_4\right)e^{h\lambda}
+\left(\frac{\alpha(\alpha-1)}{2}\left(\frac{\alpha^2}{8}+\frac{5}{24}\alpha-\gamma_4\right)\right.\\
&\left.
-\alpha\left(-\frac{\alpha^2}{4}+\frac{1}{12}\alpha+1+3\gamma_4\right)
+\left(\frac{\alpha^2}{8}-\frac{7}{24}\alpha-3\gamma_4\right)\right)e^{-h\lambda}\\
\geq&\left(\left(\frac{\alpha^2}{8}+\frac{5}{24}\alpha-\gamma_4\right)
+\frac{\alpha(\alpha-1)}{2}\left(\frac{\alpha^2}{8}+\frac{5}{24}\alpha-\gamma_4\right)\right.\\
&\left.
-\alpha\left(-\frac{\alpha^2}{4}+\frac{1}{12}\alpha+1+3\gamma_4\right)
+\left(\frac{\alpha^2}{8}-\frac{7}{24}\alpha-3\gamma_4\right)\right)e^{-h\lambda}\\
=&\left(\left(\frac{\alpha^4}{16}+\frac{7}{24}\alpha^3+\frac{\alpha^2}{16}-\frac{13}{12}\alpha\right)
-\left(\frac{\alpha^2}{2}+\frac{5}{2}\alpha+4\right)\gamma_4\right)e^{-h\lambda}.
\end{aligned}
\end{equation}
If $\gamma_4\leq\left(\frac{1}{8}\alpha^4+\frac{7}{12}\alpha^3+\frac{1}{8}\alpha^2-\frac{13}{6}\alpha\right)/(\alpha^2+5\alpha+8)$, it is immediate that $g_{0,\lambda}+g_{2,\lambda}\geq0$.

For the term $g_{1,\lambda}$, it is found that
\begin{equation}
\begin{aligned}
g_{1,\lambda}=&\gamma_1\omega_1^{(\alpha)}+\gamma_2\omega_0^{(\alpha)}
=-\alpha\left(\frac{\alpha^2}{8}+\frac{5}{24}\alpha-\gamma_4\right)+\left(-\frac{\alpha^2}{4}+\frac{1}{12}\alpha+1+3\gamma_4\right)\\
=&-\frac{\alpha^3}{8}-\frac{11}{24}\alpha^2+\frac{\alpha}{12}+1+(3+\alpha)\gamma_4.
\end{aligned}
\end{equation}
When $\gamma_4\leq\left(\frac{1}{8}\alpha^4+\frac{7}{12}\alpha^3+\frac{1}{8}\alpha^2-\frac{13}{6}\alpha\right)/(\alpha^2+5\alpha+8)$, $g_{1,\lambda}\leq0$.

For the term $g_{3,\lambda}$, it follows that,
\begin{equation}
\begin{aligned}
g_{3,\lambda}
=&\left(\gamma_1\omega_{3}^{(\alpha)}+\gamma_2\omega_{2}^{(\alpha)}
+\gamma_3\omega_{1}^{(\alpha)}+\gamma_4\omega_{0}^{(\alpha)}\right)e^{-2h\lambda}\\
=&\left(\left(\frac{\alpha^2}{8}+\frac{5}{24}\alpha-\gamma_4\right)\frac{\alpha(\alpha-1)(2-\alpha)}{6}
 +\left(-\frac{\alpha^2}{4}+\frac{\alpha}{12}+1+3\gamma_4\right)\frac{\alpha(\alpha-1)}{2}\right.\\
 &\left.
 -\left(\frac{\alpha^2}{8}-\frac{7}{24}\alpha-3\gamma_4\right)\alpha+\gamma_4
 \right)e^{-2h\lambda}\\
=&\left(-\left(\frac{\alpha^5}{48}+\frac{7}{72}\alpha^4-\frac{5}{48}\alpha^3-\frac{49}{72}\alpha^2+\frac{1}{2}\alpha\right)+\left(\frac{\alpha^3}{6}+\alpha^2+\frac{11}{6}\alpha+1\right)\gamma_4\right)e^{-2h\lambda}.
\end{aligned}
\end{equation}
If $\gamma_4\geq\left(\frac{\alpha^5}{8}+\frac{7}{12}\alpha^4-\frac{5}{8}\alpha^3-\frac{49}{12}\alpha^2+3\alpha\right)/
\left(\alpha^3+6\alpha^2+11\alpha+6\right)$, we find  $g_{3,\lambda}\geq0$.

For the term $g_{4,\lambda}$, we derive
\begin{equation}
\begin{aligned}
g_{4,\lambda}
=&\left(\gamma_1\omega_{4}^{(\alpha)}+\gamma_2\omega_{3}^{(\alpha)}+\gamma_3\omega_{2}^{(\alpha)}+\gamma_4\omega_{1}^{(\alpha)}\right)e^{-3h\lambda}\\
=&\left(\left(\frac{\alpha^2}{8}+\frac{5}{24}\alpha-\gamma_4\right)\frac{(3-\alpha)(2-\alpha)(1-\alpha)}{24}
 +\left(-\frac{\alpha^2}{4}+\frac{\alpha}{12}+1+3\gamma_4\right)\frac{(2-\alpha)(1-\alpha)}{6}\right.\\
 &\left.
 +\left(\frac{\alpha^2}{8}-\frac{7}{24}\alpha-3\gamma_4\right)\left(\frac{1-\alpha}{2}\right)
 +\gamma_4\right)\omega_{1}^{(\alpha)}e^{-3h\lambda}\\
=&-\frac{\alpha}{24}\left(-\left(\frac{\alpha^5}{8}+\frac{11}{24}\alpha^4-\frac{41}{24}\alpha^3
-\frac{107}{24}\alpha^2+\frac{163}{12}\alpha-8\right)
+\left(\alpha^3+6\alpha^2+11\alpha+6\right)\gamma_4\right)e^{-3h\lambda}.
\end{aligned}
\end{equation}
If $\gamma_4\leq\left(\frac{\alpha^5}{8}+\frac{11}{24}\alpha^4-\frac{41}{24}\alpha^3
-\frac{107}{24}\alpha^2+\frac{163}{12}\alpha-8\right)/\left(\alpha^3+6\alpha^2+11\alpha+6\right)$, then it is straightforward to find $g_{4,\lambda}\geq0$.

For the term $g_{k,\lambda},\; k\geq5$, we finally find,
\begin{equation}
\begin{aligned}
g_{k,\lambda}
=&\left(\gamma_1\omega_{k}^{(\alpha)}+\gamma_2\omega_{k-1}^{(\alpha)}+\gamma_3\omega_{k-2}^{(\alpha)}+\gamma_4\omega_{k-3}^{(\alpha)}\right)e^{-(k-1)h\lambda}\\
=&\left(\left(\frac{\alpha^2}{8}+\frac{5}{24}\alpha-\gamma_4\right)\left(\frac{k-1-\alpha}{k}\right)\left(\frac{k-2-\alpha}{k-1}\right)\left(\frac{k-3-\alpha}{k-2}\right)\right.\\
 &\left. +\left(-\frac{\alpha^2}{4}+\frac{\alpha}{12}+1+3\gamma_4\right)\left(\frac{k-2-\alpha}{k-1}\right)\left(\frac{k-3-\alpha}{k-2}\right)\right.\\
 &\left.
 +\left(\frac{\alpha^2}{8}-\frac{7}{24}\alpha-3\gamma_4\right)\left(\frac{k-3-\alpha}{k-2}\right)
 +\gamma_4\right)\omega_{k-3}^{(\alpha)}e^{-(k-1)h\lambda}\\
=&\frac{\omega_{k-3}^{(\alpha)}e^{-(k-1)h\lambda}}{k(k-1)(k-2)}\left(k^3-\left(\frac{\alpha^2}{2}+\frac{5}{2}\alpha+5\right)k^2
+\left(\frac{\alpha^4}{8}+\frac{13}{12}\alpha^3+\frac{31}{8}\alpha^2+\frac{83}{12}\alpha+6\right)k\right.\\
&\left.-\left(\frac{\alpha^5}{8}+\frac{23}{24}\alpha^4+\frac{21}{8}\alpha^3+\frac{73}{24}\alpha^2+\frac{5}{4}\alpha\right)
+\left(\alpha^3+6\alpha^2+11\alpha+6\right)\gamma_4\right).
\end{aligned}
\end{equation}
If, $$\gamma_4\geq\frac{-k^3+\left(\frac{\alpha^2}{2}+\frac{5}{2}\alpha+5\right)k^2
-\left(\frac{\alpha^4}{8}+\frac{13}{12}\alpha^3+\frac{31}{8}\alpha^2+\frac{83}{12}\alpha+6\right)k
+\left(\frac{\alpha^5}{8}+\frac{23}{24}\alpha^4+\frac{21}{8}\alpha^3+\frac{73}{24}\alpha^2+\frac{5}{4}\alpha\right)}
{\alpha^3+6\alpha^2+11\alpha+6},$$
we can find $g_{k,\lambda}\geq0$, as $\omega_{k-3}^{(\alpha)}>0,\; k\geq5$.

Next we consider $$\frac{-k^3+\left(\frac{\alpha^2}{2}+\frac{5}{2}\alpha+5\right)k^2
-\left(\frac{\alpha^4}{8}+\frac{13}{12}\alpha^3+\frac{31}{8}\alpha^2+\frac{83}{12}\alpha+6\right)k
+\left(\frac{\alpha^5}{8}+\frac{23}{24}\alpha^4+\frac{21}{8}\alpha^3+\frac{73}{24}\alpha^2+\frac{5}{4}\alpha\right)}
{\alpha^3+6\alpha^2+11\alpha+6},$$
and analyze the function,
 \begin{eqnarray}
  g(x)&=&-x^3+\left(\frac{\alpha^2}{2}+\frac{5}{2}\alpha+5\right)x^2-\left(\frac{\alpha^4}{8}+\frac{13}{12}\alpha^3+\frac{31}{8}\alpha^2+\frac{83}{12}\alpha+6\right)x\nonumber\\
  &+&\left(\frac{\alpha^5}{8}+\frac{23}{24}\alpha^4+\frac{21}{8}\alpha^3+\frac{73}{24}\alpha^2+\frac{5}{4}\alpha\right).
 \end{eqnarray}
Function $g(x)$ is monotonically decreasing as the variable $x\ (x\geq5)$ for $1.26<\alpha<1.71.$ When $\alpha^3+6\alpha^2+11\alpha+6\geq0$, we have
\begin{eqnarray}
\frac{g(x)}{\alpha^3+6\alpha^2+11\alpha+6}
&\leq&\frac{g(5)}{\alpha^3+6\alpha^2+11\alpha+6}\nonumber\\
&=&\frac{\frac{1}{8}\alpha^5 + \frac{1}{3} \alpha^4- \frac{ 67}{24}\alpha^3 - \frac{ 23}{6}\alpha^2 +\frac{175}{6}\alpha - 30}{\alpha^3+6\alpha^2+11\alpha+6}.
\end{eqnarray}
\end{proof}

We then analyze the generating functions $f(\alpha,\lambda;x)$ of $H$ given in \eqref{eq:fh}.
\begin{theorem}\label{l3-9}
Let the matrices $B_{\lambda}$ and $B_{\lambda}^T$ be given by \eqref{e3-5}. For $\lambda\geq0,\ h>0$ and $\alpha\in(1.26,1.71)$, $f(\alpha;x)$ is the generating function of $H=\frac{B_{\lambda}+B_{\lambda}^T}{2}$, if $\gamma_4\in(a_3,a_4)$, we have $f(\alpha;x)<0$ and $B_{\lambda}$ is negative.
\end{theorem}
\begin{proof}
We consider the function $\phi(\lambda)$, and obtain,
\begin{equation}
\begin{aligned}
\phi(\lambda)=&
\left(\gamma_1e^{h\lambda}+\gamma_2+\gamma_3e^{-h\lambda}+\gamma_4e^{-2h\lambda}\right)(1-e^{-h\lambda})^\alpha\\
=&\left((\frac{\alpha^2}{8}+\frac{5}{24}\alpha-\gamma_4)e^{h\lambda}
+(-\frac{\alpha^2}{4}+\frac{1}{12}\alpha+1+3\gamma_4)
+(\frac{\alpha^2}{8}-\frac{7}{24}\alpha-3\gamma_1)e^{-h\lambda}
+\gamma_4e^{-2h\lambda}\right)(1-e^{-h\lambda})^\alpha\\
=&\left(\frac{\alpha^2}{8}(e^{h\lambda}+e^{-h\lambda}-2)
+\frac{\alpha}{24}(5e^{h\lambda}-7e^{-h\lambda}+2)+1
+\gamma_4(-e^{h\lambda}+3-3e^{-h\lambda}+e^{-2h\lambda})
\right)(1-e^{-h\lambda})^\alpha\\
\geq&\gamma_4\left(-e^{h\lambda}+3-3e^{-h\lambda}+e^{-2h\lambda}
\right)(1-e^{-h\lambda})^\alpha.
\end{aligned}
\end{equation}
Denoting $h(x)=-e^{x}+3-3e^{-x}+e^{-2x}$, we obtain $h(x)\leq h(0)=0$, as $h(x)$ is monotonically decreasing when $x\geq0$. It is found that $\phi(\lambda)>0 $ as $\gamma_4<0$.

By Theorem \ref{l3-8}, we obtain
\begin{equation}
\begin{aligned}
f(\alpha;x)=&\sum_{k=0}^{N}g_{k,\lambda}^{\alpha}\cos(k-1)x-\phi(\lambda)
\leq\sum_{k=0}^{N}g_{k,\lambda}^{\alpha}-\phi(\lambda)\\
=&\sum_{k=0}^{N}g_{k,\lambda}^{\alpha}-\sum_{k=0}^{+\infty}g_{k,\lambda}^{\alpha}
<0.
\end{aligned}
\end{equation}
By Lemma \ref{l3-6}, we see that $B_{\lambda}$ is negative.

\end{proof}

For any $\alpha \in (1.26,1.71)$, if $\gamma_4\in(a_3,a_4)$, the numerical scheme \eqref{4-2-2} will be unconditionally stable by Corollary \ref{T3-8} and Theorem \ref{l3-9}. If $\gamma_4\notin(a_3,a_4)$, we can also obtain a stable numerical scheme \eqref{4-2-2} whenever $f(\alpha,\lambda;x)<0$. This condition can be evaluated by numerical examples in Section \ref{sec:example}.

In fact, $\forall\ \alpha\in(1,2)$, if a $\gamma_i$ exists which satisfies $f(\alpha,\lambda;x)<0$ for certain $N>0$, the numerical scheme \eqref{4-2-2} will be stable. The numerical examples in Section \ref{sec:example} confirm this. We will check whether $\gamma_4$ can be obtained so that $f(\alpha,\lambda;x)<0$ by a numerical experiment in Figures \ref{fig:3-1-1} and \ref{fig:3-1-2}. Choosing $\alpha=1.2\in(1,1.26)$ in Figure \ref{fig:3-1-1} and $\alpha=1.8\in(1.71,2)$ in Figure \ref{fig:3-1-2}, it is found that for two values of $\alpha\notin(1.26,1.71),$ $\gamma_4$ can be numerically found so that $f(\alpha,\lambda;x)<0$ for certain $N$-values.
\begin{figure}[!htpb]
\begin{center}
    \includegraphics[scale=0.22]{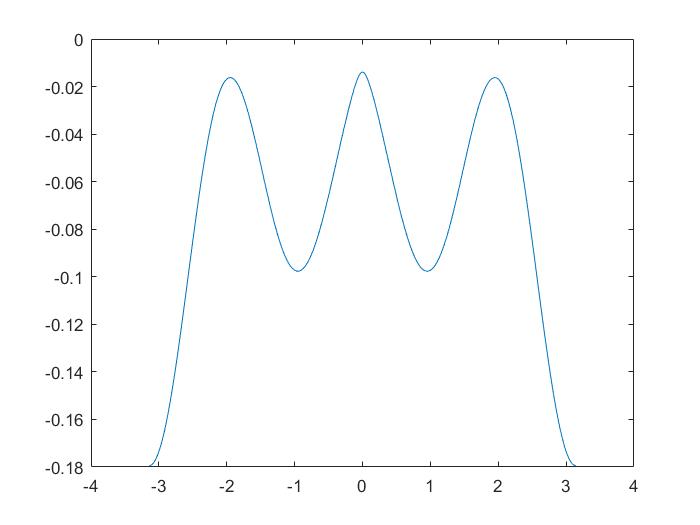}\hspace{-0.5cm}
    \includegraphics[scale=0.22]{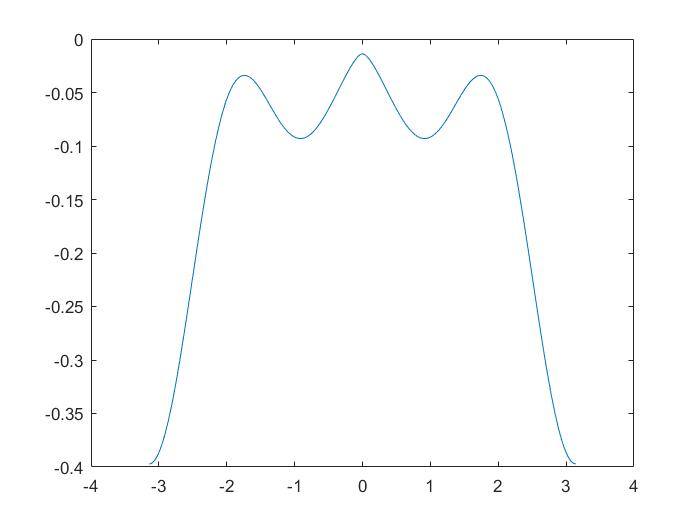}\hspace{-0.5cm}
    \includegraphics[scale=0.22]{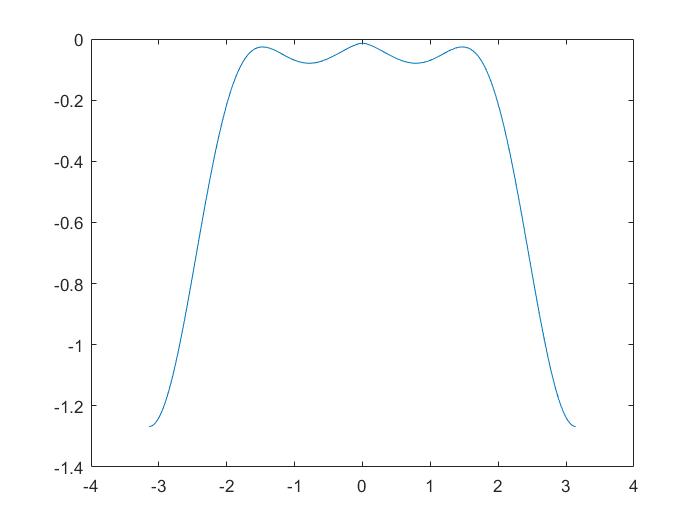}\hspace{-0.5cm}\\
    $\gamma_4=-0.075$\hspace{3.65cm}$\gamma_4=-0.1$\hspace{3.65cm} $\gamma_4=-0.2$
\end{center}
\caption{Function $f(\alpha;x)$ with $\alpha=1.2$, $\lambda=2$ and $N=250$}
\label{fig:3-1-1}
\end{figure}
\begin{figure}[!htpb]
\begin{center}
    \includegraphics[scale=0.22]{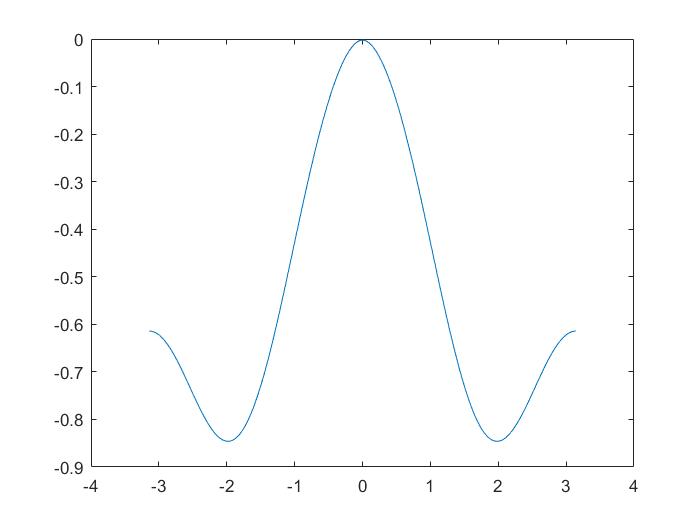}\hspace{-0.5cm}
    \includegraphics[scale=0.22]{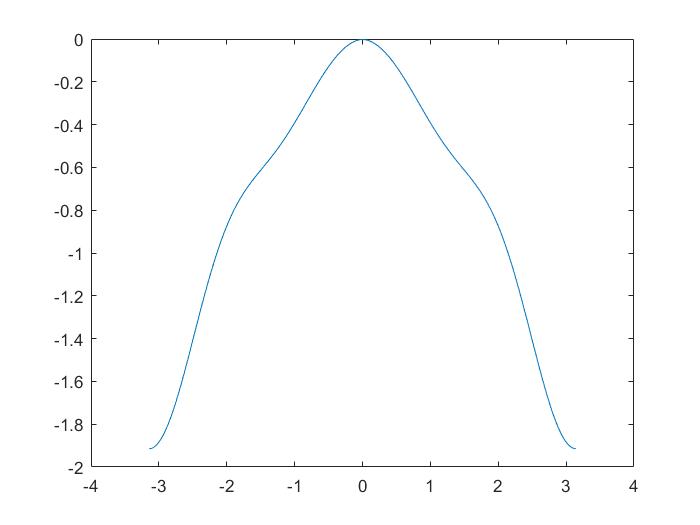}\hspace{-0.5cm}
    \includegraphics[scale=0.22]{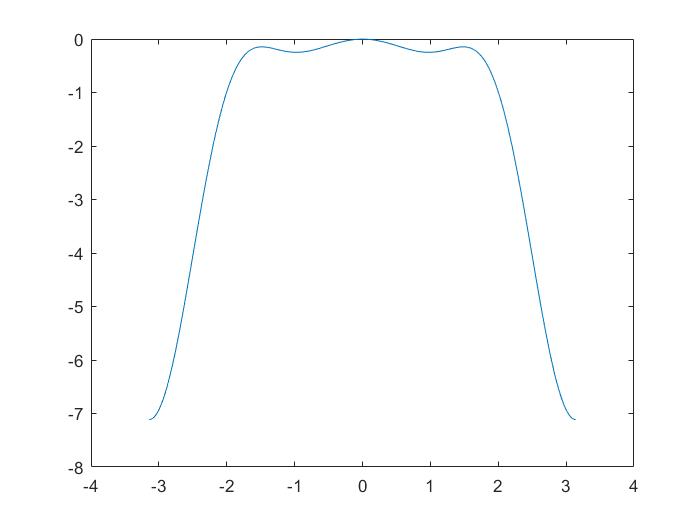}\hspace{-0.5cm}\\
    $\gamma_4=0$\hspace{3.5cm}$\gamma_4=-0.1$\hspace{3.5cm} $\gamma_4=-0.5$
\end{center}
\caption{Function $f(\alpha;x)$ with $\alpha=1.8$, $\lambda=2$ and $N=250$}
\label{fig:3-1-2}
\end{figure}



Error estimates for the fully discrete scheme \eqref{4-2-2} are based on the following lemma.
\begin{lemma}\label{l-11}(Discrete Gronwall's Inequation, see, for example, \cite{Alfio1997Num})
Assume that $\{k_n\}$ and $\{p_n\}$ are nonnegative sequences, and the sequence $\{ \phi\}$ satisfies
$$\phi_0\leq g_0,\ \phi_n\leq g_0+\sum_{l=0}^{n-1}p_l+\sum_{l=0}^{n-1}k_l\phi_l, \ n\geq1,$$
where $g_0\leq0.$ Then, the sequence $\{ \phi\}$ satisfies
$$\phi_n\leq\left(g_0+\sum_{l=1}^{n-1}p_l\right)\exp\left(\sum_{l=1}^{n-1}k_l\right),\ n\geq1.$$
\end{lemma}
\begin{theorem}\label{T3-9}
Let's denote by $e_i^j=u^j_i-U^j_i, i=1,2,...,M-1,$ and $E^j=(e^j_1,e^j_2,...,e^j_{M-1})^T, j=1,2,...,N$. With solutions $u^j_i$ and $U^j_i$ of Equations \eqref{4-2-1} and \eqref{4-2-2}, respectively, we have, for $1<\alpha<2$, if $f(\alpha,\lambda_i;x)<0$, $i=1,2$,
\begin{equation}
\norm{E^j}_h \leq c(\tau^2+h^3),\ 1\leq j\leq N-1.
\end{equation}
\end{theorem}

\begin{proof}
Combining \eqref{4-2-1} and \eqref{4-2-2}, gives us,
\begin{equation}
\begin{aligned}
&e_{i}^{j+1}
-\frac{\tau}{2}\left[
c_l\cdot~ _L\mathcal{D}^{\alpha,\lambda_1}_{h,3}e_{i}^{j+1}
+c_r\cdot~_R\mathcal{D}^{\alpha,\lambda_2}_{h,3}e_{i}^{j+1}
\right]\\
&\quad=e_{i}^j+\frac{\tau}{2}\left[
c_l\cdot~ _L\mathcal{D}^{\alpha,\lambda_1}_{h,3}e_{i}^{j}
+c_r\cdot~ _R\mathcal{D}^{\alpha,\lambda_2}_{h,3}e_{i}^{j}
\right]
+\tau\rho_i^j,
\end{aligned}
\end{equation}
where $\rho_i^j=O(\tau^2+ h^3)$.
In matrix form, this is given by
\begin{equation}
\label{4-20}
\begin{aligned}
(I-\mathcal{M})E^{j+1}
=(I+\mathcal{M})E^j
+\tau\rho^j,
\end{aligned}
\end{equation}
where $\rho^j=(\rho^j_1,\rho^j_2,...,\rho^j_{M-1})^T.$\\
After multiplication on both sides of Equation \eqref{4-20} by $(E^{j+1}+E^{j})^T$, we have
\begin{equation}
\begin{aligned}
(E^{j+1}+E^{j})^T(E^{j+1}-E^{j})
-(E^{j+1}+E^{j})^T\mathcal{M}(E^{j+1}+E^{j})=\tau(E^{j+1}+E^{j})^T\rho^j&.
\end{aligned}
\end{equation}
By Theorem \ref{T3-8}, it is known that $\mathcal{M}$ is negative. So, we obtain,
\begin{equation}
\begin{aligned}
(E^{j+1}+E^{j})^T\mathcal{M}(E^{j+1}+E^{j})<0.
\end{aligned}
\end{equation}
Hence, we obtain
\begin{equation}
\begin{aligned}
(E^{j+1}+E^{j})^T(E^{j+1}-E^{j})
&=\sum_{i=1}^{M-1}((e_{i}^{j+1})^2-(e_{i}^{j})^2)\\
\leq\tau(E^{j+1}+E^{j})^T\rho^j&=\tau\sum_{i=1}^{M-1}(e_{i}^{j+1}+e_{i}^{j})\rho^j_i.
\end{aligned}
\end{equation}
Summing up for all $j\in[0,n-1]$, we conclude that
\begin{equation}
\begin{aligned}
\sum_{i=1}^{M-1}(e_{i}^{n})^2
&\leq\tau\sum_{j=0}^{n-1}\sum_{i=1}^{M-1}(e_{i}^{j+1}+e_{i}^{j})\rho^j_i\\
&=\tau\sum_{j=1}^{n-1}\sum_{i=1}^{M-1}(\rho_{i}^{j}+\rho_{i}^{j-1})e^j_i+\tau\sum_{i=1}^{M-1}e_{i}^{n}\rho_i^{n-1}\\
&\leq\frac{\tau}{2}\sum_{j=0}^{n-1}\sum_{i=1}^{M-1}(e^j_i)^2
+\frac{\tau}{2}\sum_{j=0}^{n-1}\sum_{i=1}^{M-1}(\rho_{i}^{j}+\rho_{i}^{j-1})^2
+\frac{1}{2}\sum_{i=1}^{M-1}(e_{i}^{n})^2+\frac{1}{2}\sum_{i=1}^{M-1}(\tau\rho_i^{n-1})^2.
\end{aligned}
\end{equation}
As $\rho_i^j=O(\tau^2+ h^3)$, we have the following result by the discrete Gronwall's inequality \eqref{l-11},
\begin{equation}
\begin{aligned}
\norm{E^{n}}^2\leq\tau\sum_{j=1}^{n-1}
\norm{E^{j}}^2+c(\tau^2+ h^3)^2
\leq e^Tc(\tau^2+ h^3)^2
\leq C(\tau^2+ h^3)^2.
\end{aligned}
\end{equation}
\end{proof}

\section{Numerical examples}\label{sec:example}
In this section, we present some numerical results for several experiments, on tempered fractional derivatives in example \ref{e1}, the tempered fractional diffusion equations in example \ref{e2} and the tempered fractional Black--Scholes equation in example \ref{e3}-\ref{e4}, to verify the theoretical results.
\subsection{The tempered fractional derivatives}
In this subsection, we take the second-order operators \eqref{2-22-1}-\eqref{2-23-1} and third-order operators \eqref{2-26-1}-\eqref{2-27-1} for the left and right tempered fractional derivatives to test the accuracy of the tempered-WSGD operators.
\begin{example}\label{e1}
In the example, we choose $\alpha=0.6$ and $\alpha=1.6$, and consider different $\lambda$ in the interval $[0,1]$ for the left and right tempered fractional derivatives. 
\begin{enumerate}
  \item We analyze the schemes for the following left tempered fractional derivative, $$_0\mathfrak{D}_x^{\alpha,\lambda}(e^{-\lambda x}x^{3+\alpha})=e^{-\lambda x}x^3\left(\frac{\Gamma(4+\alpha)}{6}-\lambda^\alpha x^{\alpha}\right),$$ which is discretized by the second-order operator \eqref{2-22-1} and the third-order scheme \eqref{2-26-1}, respectively. Tables \ref{tab:4-1-1} and \ref{tab:4-1-2} show the corresponding $L^2$ errors and the orders of accuracy for different $\lambda$-values, with $\alpha=0.6,$ and $\alpha=1.6$, $\gamma_3=0.001$ for the second-order operator, and $\gamma_4=0.001$ for the third-order operator. 
      The results confirm the desired accuracy.
\begin{table}[!ht]\centering
\begin{tabular}{c c c c c c c c}
\toprule
&   &  \multicolumn{2}{c}{$\lambda=0$}    & \multicolumn{2}{c}{$\lambda=1$}        & \multicolumn{2}{c}{$\lambda=3$}   \\
\cmidrule{3-8}
$\alpha$&   h &      \multicolumn{1}{c}{$L^{2}$ Error } &  \multicolumn{1}{c}{Order}   & \multicolumn{1}{c}{$L^{2}$ Error }  & \multicolumn{1}{c}{Order}  &  \multicolumn{1}{c}{$L^{2}$ Error}  & \multicolumn{1}{c}{Order}\\
\midrule
            &	$1/10\ $&	1.05e-02	&		&	4.87e-03	&		&	1.30e-03	&		\\
            &	$1/20\ $&	2.58e-03	&	2.03	&	1.21e-03	&	2.01	&	3.25e-04	&	1.99	\\
$\alpha=0.6$&	$1/40\ $&	6.39e-04	&	2.01	&	3.02e-04	&	2.00	&	8.15e-05	&	2.00	\\
            &	$1/80\ $&	1.59e-04	&	2.01	&	7.55e-05	&	2.00	&	2.04e-05	&	2.00	\\
            &	$1/160 $&	3.96e-05	&	2.00	&	1.89e-05	&	2.00	&	5.11e-06	&	2.00	\\
\cmidrule{1-8}
            &$1/10\ $&	5.29e-02	&		&	2.50e-02	&		&	7.02e-03	&		\\
            &$1/20\ $&	1.32e-02	&	2.01	&	6.33e-03	&	1.98	&	1.83e-03	&	1.94	\\
$\alpha=1.6$&$1/40\ $&	3.29e-03	&	2.00	&	1.59e-03	&	1.99	&	4.68e-04	&	1.97	\\
            &$1/80\ $&	8.21e-04	&	2.00	&	4.00e-04	&	1.99	&	1.18e-04	&	1.98	\\
            &$1/160 $&	2.05e-04	&	2.00	&	1.00e-04	&	2.00	&	2.97e-05	&	1.99	\\

\bottomrule
\end{tabular}
\caption{$L^2$ errors and orders of accuracy for $_0\mathfrak{D}_x^{\alpha,\lambda}(e^{-\lambda x}x^{3+\alpha})$ by the second-order operator \eqref{2-22-1} for different $\lambda$ with $\gamma_3=0.001$.}
\label{tab:4-1-1}
\end{table}

\begin{table}[!ht]\centering
\begin{tabular}{c c c c c c c c}
\toprule
&   &  \multicolumn{2}{c}{$\lambda=0$}    & \multicolumn{2}{c}{$\lambda=1$}        & \multicolumn{2}{c}{$\lambda=3$}   \\
\cmidrule{3-8}
$\alpha$&   h &      \multicolumn{1}{c}{$L^{2}$ Error } &  \multicolumn{1}{c}{Order}   & \multicolumn{1}{c}{$L^{2}$ Error }  & \multicolumn{1}{c}{Order}  &  \multicolumn{1}{c}{$L^{2}$ Error}  & \multicolumn{1}{c}{Order}\\
\midrule
            &	$1/10\ $&	6.62e-04	&		&	4.09e-04	&		&	2.20e-04	&		\\
            &	$1/20\ $&	8.48e-05	&	2.96	&	5.38e-05	&	2.92	&	3.08e-05	&	2.84	\\
$\alpha=0.6$&	$1/40\ $&	1.07e-05	&	2.98	&	6.91e-06	&	2.96	&	4.06e-06	&	2.92	\\
            &	$1/80\ $&	1.35e-06	&	2.99	&	8.75e-07	&	2.98	&	5.22e-07	&	2.96	\\
            &	$1/160 $&	1.69e-07	&	3.00	&	1.10e-07	&	2.99	&	6.62e-08	&	2.98	\\

\cmidrule{1-8}
&$1/10\ $&	5.79e-03	&		&	3.54e-03	&		&	1.84e-03	&		\\
&$1/20\ $&	7.47e-04	&	2.95	&	4.73e-04	&	2.90	&	2.66e-04	&	2.79	\\
$\alpha=1.6$&$1/40\ $&	9.49e-05	&	2.98	&	6.12e-05	&	2.95	&	3.59e-05	&	2.89	\\
&$1/80\ $&	1.20e-05	&	2.99	&	7.78e-06	&	2.98	&	4.66e-06	&	2.95	\\
&$1/160 $&	1.50e-06	&	2.99	&	9.82e-07	&	2.99	&	5.93e-07	&	2.97	\\

\bottomrule
\end{tabular}
\caption{$L^2$ errors and orders of accuracy for $_0\mathfrak{D}_x^{\alpha,\lambda}(e^{-\lambda x}x^{3+\alpha})$ by the third-order operator \eqref{2-26-1} for different $\lambda$ with $\gamma_4=0.001$.}
\label{tab:4-1-2}
\end{table}
\item We also consider the right tempered fractional derivative, $$_x\mathfrak{D}_1^{\alpha,\lambda}(e^{\lambda x}(1-x)^{3+\alpha})=
e^{\lambda x}(1-x)^3\left(\frac{\Gamma(4+\alpha)}{6}-\lambda^\alpha(1-x)^{\alpha}\right),$$ computed by the second-order operator \eqref{2-23-1} and the third-order operators \eqref{2-27-1}, respectively. Tables \ref{tab:4-1-3} and \ref{tab:4-1-4} show the $L^2$ errors and orders of accuracy for different $\lambda$-values, with $\gamma_3=-0.001$ for the second-order and $\gamma_4=-0.001$ for the third-order operators. 
The results clearly confirm the desired discretization accuracy.
\begin{table}[!ht]\centering
\begin{tabular}{c c c c c c c c}
\toprule
&   &  \multicolumn{2}{c}{$\lambda=0$}    & \multicolumn{2}{c}{$\lambda=1$}        & \multicolumn{2}{c}{$\lambda=3$}   \\
\cmidrule{3-8}
$\alpha$&   h &      \multicolumn{1}{c}{$L^{2}$ Error } &  \multicolumn{1}{c}{Order}   & \multicolumn{1}{c}{$L^{2}$ Error }  & \multicolumn{1}{c}{Order}  &  \multicolumn{1}{c}{$L^{2}$ Error}  & \multicolumn{1}{c}{Order}\\
\midrule
&	$1/10\ $&	8.94e-03	&		&	1.21e-02	&		&	2.57e-02	&		\\
&	$1/20\ $&	2.36e-03	&	1.92	&	3.13e-03	&	1.95	&	6.44e-03	&	2.00	\\
$\alpha=0.6$&	$1/40\ $&	6.06e-04	&	1.96	&	7.94e-04	&	1.98	&	1.61e-03	&	2.00	\\
&	$1/80\ $&	1.54e-04	&	1.98	&	2.00e-04	&	1.99	&	4.04e-04	&	2.00	\\
&	$1/160 $&	3.87e-05	&	1.99	&	5.03e-05	&	1.99	&	1.01e-04	&	2.00	\\

\cmidrule{1-8}
&$1/10\ $&	4.49e-02	&		&	6.17e-02	&		&	1.40e-01	&		\\
&$1/20\ $&	1.21e-02	&	1.90	&	1.63e-02	&	1.92	&	3.63e-02	&	1.94	\\
$\alpha=1.6$&$1/40\ $&	3.12e-03	&	1.95	&	4.19e-03	&	1.96	&	9.27e-03	&	1.97	\\
&$1/80\ $&	7.95e-04	&	1.97	&	1.06e-03	&	1.98	&	2.34e-03	&	1.98	\\
&$1/160 $&	2.00e-04	&	1.99	&	2.67e-04	&	1.99	&	5.89e-04	&	1.99	\\

\bottomrule
\end{tabular}
\caption{$L^2$ errors and orders of accuracy for $_x\mathfrak{D}_1^{\alpha,\lambda}(e^{\lambda x}(1-x)^{3+\alpha})$ by the second-order operator \eqref{2-23-1} for different $\lambda$ with $\gamma_3=-0.001$.}
\label{tab:4-1-3}
\end{table}
\begin{table}[!ht]\centering
\begin{tabular}{c c c c c c c c}
\toprule
&   &  \multicolumn{2}{c}{$\lambda=0$}    & \multicolumn{2}{c}{$\lambda=1$}        & \multicolumn{2}{c}{$\lambda=3$}   \\
\cmidrule{3-8}
$\alpha$&   h &      \multicolumn{1}{c}{$L^{2}$ Error } &  \multicolumn{1}{c}{Order}   & \multicolumn{1}{c}{$L^{2}$ Error }  & \multicolumn{1}{c}{Order}  &  \multicolumn{1}{c}{$L^{2}$ Error}  & \multicolumn{1}{c}{Order}\\
\midrule
&$1/10\ $&	6.53e-04	&		&	1.14e-03	&		&	4.63e-03	&		\\
&$1/20\ $&	8.59e-05	&	2.93	&	1.51e-04	&	2.92	&	6.44e-04	&	2.85	\\
$\alpha=0.6$&$1/40\ $&	1.10e-05	&	2.96	&	1.95e-05	&	2.96	&	8.49e-05	&	2.92	\\
&$1/80\ $&	1.39e-06	&	2.98	&	2.47e-06	&	2.98	&	1.09e-05	&	2.96	\\
&$1/160 $&	1.75e-07	&	2.99	&	3.11e-07	&	2.99	&	1.38e-06	&	2.98	\\

\cmidrule{1-8}
&$1/10\ $&	5.58e-03	&		&	9.62e-03	&		&	3.77e-02	&		\\
&$1/20\ $&	7.42e-04	&	2.91	&	1.30e-03	&	2.89	&	5.44e-03	&	2.79	\\
$\alpha=1.6$&$1/40\ $&	9.56e-05	&	2.96	&	1.69e-04	&	2.94	&	7.34e-04	&	2.89	\\
&$1/80\ $&	1.21e-05	&	2.98	&	2.15e-05	&	2.97	&	9.53e-05	&	2.94	\\
&$1/160 $&	1.53e-06	&	2.99	&	2.72e-06	&	2.99	&	1.21e-05	&	2.97	\\

\bottomrule
\end{tabular}
\caption{$L^2$ errors and orders of accuracy for $_x\mathfrak{D}_1^{\alpha,\lambda}(e^{\lambda x}(1-x)^{3+\alpha})$ by the third-order operator \eqref{2-27-1} for different $\lambda$ with $\gamma_4=-0.001$.}
\label{tab:4-1-4}
\end{table}

\end{enumerate}
\end{example}
\subsection{The tempered fractional diffusion equation}
In this subsection, we numerically test the accuracy of the second-order scheme \eqref{4-1-2} and the third-order scheme \eqref{4-2-2} for the tempered fractional diffusion equations.
\begin{example}\label{e2}
In this example, however, with $x\in[0,1]$, we choose different $\alpha$ for three different tempered fractional diffusion equations. To ensure the stability, we take $\gamma_3\in(a_1,a_2)$ for the second-order scheme \eqref{4-1-2}, and $\gamma_4\in(a_3,a_4)$ or $\gamma_4$ satisfying $f(\alpha;x)<0$ for the third-order scheme \eqref{4-2-2}. The numerical results are shown in Table \ref{tab:4-2-1}-\ref{tab:4-2-6} which confirm the desired accuracy. 
\begin{enumerate}
  \item
 We consider the following tempered fractional diffusion equation with the left tempered fractional derivative,
  \begin{equation}
  \label{eq:2-2}
  \frac{\partial u(x,t)}{\partial t}=_0D_x^{\alpha,\lambda}u(x,t)-e^{-\lambda x-t}\left(x^{3+\alpha}+\frac{\Gamma(4+\alpha)}{6}x^3\right),
  \end{equation}
  with the boundary conditions
  \begin{equation}
  \label{eq:2-2a}
  u(0,t)=0,\ u(1,t)=e^{-\lambda-t},\ t\in[0,1],
  \end{equation}
  and the initial value
  \begin{equation}
  \label{eq:2-2b}
  u(x,0)=e^{-\lambda x}x^{1+\alpha},\ x\in[0,1].
  \end{equation}
  The exact solution for \eqref{eq:2-2} is given by $u(x,t)=e^{-\lambda x-t}x^{3+\alpha}$.

  Let $\alpha=1.2\notin(1.26,1.71)$ for the tempered fractional diffusion equation \eqref{eq:2-2}. For the second-order scheme \eqref{4-1-2}, we choose $\gamma_3\in(a_1,a_2)$ to ensure the stability. For the third-order scheme \eqref{4-2-2}, we define $\gamma_4$ in such a way that $f(\alpha;x)<0$. Because of this latter choice, we expect a stable and accurate discretization, despite the fact that $\alpha\notin(1.26,1.71)$. Take $\tau=h,\ \lambda=4.$ We construct the second-order scheme and the third-order scheme for this Equation \eqref{eq:2-2}. Tables \ref{tab:4-2-1} and \ref{tab:4-2-2} show $L^2$ errors and orders of accuracy for both schemes, confirming the desired accuracy, in both cases.

\begin{table}[!ht]\centering
\begin{tabular}{c c c c c c c}
\toprule
&\multicolumn{2}{c}{$\gamma_3=-0.01$}    & \multicolumn{2}{c}{$\gamma_3=0$}        & \multicolumn{2}{c}{$\gamma_3=0.01$}   \\
\cmidrule{2-7}
    h &      \multicolumn{1}{c}{$L^{2}$ Error } &  \multicolumn{1}{c}{Order}   & \multicolumn{1}{c}{$L^{2}$ Error }  & \multicolumn{1}{c}{Order}  &  \multicolumn{1}{c}{$L^{2}$ Error}  & \multicolumn{1}{c}{Order}\\
\midrule
$1/10\ $&	2.20e-04	&		&	2.30e-04	&		&	2.40e-04	&		\\
$1/20\ $&	5.44e-05	&	2.02	&	5.73e-05	&	2.01	&	6.02e-05	&	2.00	\\
$1/40\ $&	1.33e-05	&	2.04	&	1.40e-05	&	2.03	&	1.48e-05	&	2.03	\\
$1/80\ $&	3.25e-06	&	2.03	&	3.45e-06	&	2.02	&	3.65e-06	&	2.02	\\
$1/160 $&	8.06e-07	&	2.01	&	8.56e-07	&	2.01	&	9.05e-07	&	2.01	\\

\bottomrule
\end{tabular}
\caption{$L^2$ errors and orders of accuracy for \eqref{eq:2-2} with the boundary conditions \eqref{eq:2-2a} and the initial value \eqref{eq:2-2b} by the second-order scheme \eqref{4-1-2}.}
\label{tab:4-2-1}
\end{table}

  \begin{table}[!ht]\centering
\begin{tabular}{c c c c c c c}
\toprule
&\multicolumn{2}{c}{$\gamma_4=-0.1$}    & \multicolumn{2}{c}{$\gamma_4=-0.2$}        & \multicolumn{2}{c}{$\gamma_4=-0.25$}   \\
\cmidrule{2-7}
    h &      \multicolumn{1}{c}{$L^{2}$ Error } &  \multicolumn{1}{c}{Order}   & \multicolumn{1}{c}{$L^{2}$ Error }  & \multicolumn{1}{c}{Order}  &  \multicolumn{1}{c}{$L^{2}$ Error}  & \multicolumn{1}{c}{Order}\\
\midrule
$1/10\ $&	7.55e-05	&		&	1.08e-04	&		&	1.25e-04	&		\\
$1/20\ $&	9.44e-06	&	3.00	&	1.41e-05	&	2.94	&	1.66e-05	&	2.92	\\
$1/40\ $&	1.13e-06	&	3.07	&	1.71e-06	&	3.04	&	2.05e-06	&	3.02	\\
$1/80\ $&	1.37e-07	&	3.04	&	2.09e-07	&	3.04	&	2.51e-07	&	3.03	\\
$1/160 $&	1.77e-08	&	2.96	&	2.61e-08	&	3.00	&	3.14e-08	&	3.00	\\

\bottomrule
\end{tabular}
\caption{$L^2$ errors and orders of accuracy for \eqref{eq:2-2} with the boundary conditions \eqref{eq:2-2a} and the initial value \eqref{eq:2-2b} by the third-order scheme \eqref{4-2-2}.}
\label{tab:4-2-2}
\end{table}

  \item
  We also consider the following tempered fractional diffusion equation with the right tempered fractional derivative
  \begin{equation}
  \label{eq:3-2}
  \frac{\partial u(x,t)}{\partial t}=_0D_x^{\alpha,\lambda}u(x,t)-e^{\lambda x-t}\left((1-x)^{3+\alpha}+\frac{\Gamma(4+\alpha)}{6}(1-x)^3\right),
  \end{equation}
  with the boundary conditions
    \begin{equation}
  \label{eq:3-2a}
  u(0,t)=e^{-t},\ u(1,t)=0,\ t\in[0,1],
  \end{equation}
  and the initial value
  \begin{equation}
  \label{eq:3-2b}
  u(x,0)=e^{\lambda x}(1-x)^{3+\alpha},\ x\in[0,1].
  \end{equation}
  The exact solution for \eqref{eq:3-2} is given by $u(x,t)=e^{\lambda x-t}(1-x)^{3+\alpha}$.

Let $\alpha=1.8\notin(1.26,1.71)$ for the tempered fractional diffusion equation \eqref{eq:3-2}. We choose $\gamma_3\in(a_1,a_2)$ for the second-order scheme \eqref{4-1-2}, and define $\gamma_4$ satisfying $f(\alpha;x)<0$ for the third-order scheme \eqref{4-2-2}. Take $\tau=h,\ \lambda=4.$ Tables \ref{tab:4-2-3} and \ref{tab:4-2-4} confirm the desired $L^2$ errors and orders of accuracy for both schemes.
  \begin{table}[!ht]\centering
\begin{tabular}{c c c c c c c}
\toprule
&\multicolumn{2}{c}{$\gamma_3=-0.03$}    & \multicolumn{2}{c}{$\gamma_3=-0.02$}        & \multicolumn{2}{c}{$\gamma_3=-0.01$}   \\
\cmidrule{2-7}
    h &      \multicolumn{1}{c}{$L^{2}$ Error } &  \multicolumn{1}{c}{Order}   & \multicolumn{1}{c}{$L^{2}$ Error }  & \multicolumn{1}{c}{Order}  &  \multicolumn{1}{c}{$L^{2}$ Error}  & \multicolumn{1}{c}{Order}\\
\midrule
$1/10\ $&	1.41e-02	&		&	1.48e-02	&		&	1.55e-02	&		\\
$1/20\ $&	3.51e-03	&	2.00	&	3.78e-03	&	1.97	&	4.04e-03	&	1.94	\\
$1/40\ $&	8.65e-04	&	2.02	&	9.44e-04	&	2.00	&	1.02e-03	&	1.98	\\
$1/80\ $&	2.14e-04	&	2.02	&	2.35e-04	&	2.00	&	2.57e-04	&	1.99	\\
$1/160 $&	5.31e-05	&	2.01	&	5.87e-05	&	2.00	&	6.44e-05	&	2.00	\\

\bottomrule
\end{tabular}
\caption{$L^2$ errors and orders of accuracy for \eqref{eq:3-2} with the boundary conditions \eqref{eq:3-2a} and the initial value \eqref{eq:3-2b} by the second-order scheme \eqref{4-1-2}.}
\label{tab:4-2-3}
\end{table}

  \begin{table}[!ht]\centering
\begin{tabular}{c c c c c c c}
\toprule
&\multicolumn{2}{c}{$\gamma_4=0.01$}    & \multicolumn{2}{c}{$\gamma_4=0.02$}        & \multicolumn{2}{c}{$\gamma_4=0.03$}   \\
\cmidrule{2-7}
    h &      \multicolumn{1}{c}{$L^{2}$ Error } &  \multicolumn{1}{c}{Order}   & \multicolumn{1}{c}{$L^{2}$ Error }  & \multicolumn{1}{c}{Order}  &  \multicolumn{1}{c}{$L^{2}$ Error}  & \multicolumn{1}{c}{Order}\\
\midrule
$1/10\ $&	7.99e-03	&		&	7.62e-03	&		&	7.29e-03	&		\\
$1/20\ $&	1.19e-03	&	2.75	&	1.11e-03	&	2.78	&	1.04e-03	&	2.81	\\
$1/40\ $&	1.61e-04	&	2.88	&	1.47e-04	&	2.91	&	1.35e-04	&	2.94	\\
$1/80\ $&	2.10e-05	&	2.94	&	1.90e-05	&	2.95	&	1.72e-05	&	2.98	\\
$1/160 $&	2.72e-06	&	2.95	&	2.45e-06	&	2.96	&	2.19e-06	&	2.97	\\

\bottomrule
\end{tabular}
\caption{$L^2$ errors and orders of accuracy for \eqref{eq:3-2} with the boundary conditions \eqref{eq:3-2a} and the initial value \eqref{eq:3-2b} by the third-order scheme \eqref{4-2-2}.}
\label{tab:4-2-4}
\end{table}

\item We also analyze the following tempered fractional advection-diffusion equation
  \begin{equation}
  \label{eq:2-20}
  \frac{\partial u(x,t)}{\partial t}=c_l\left( _0D_x^{\alpha,\lambda}u(x,t)\right)+c_r\left( _xD_1^{\alpha,\lambda}u(x,t)\right)+f(x,t),
  \end{equation}
  with the boundary conditions,
  \begin{equation}
  \label{eq:2-20a}
  u(0,t)=0,\ u(1,t)=0,\ t\in[0,1],
  \end{equation}
  and the initial value,
  \begin{equation}
  \label{eq:2-20b}
  u(x,0)=e^{-\lambda x}x^{1+\alpha},\ x\in[0,1].
  \end{equation}
  With
  \begin{equation}\nonumber
  \begin{aligned}
  f=&-e^{-t}x^{3}(1-x)^{3}-e^{-\lambda x-t}\left( _0D_x^{\alpha}\Big[e^{\lambda x}(x^3-3x^4+3x^5-x^6)\Big]\right)\\
  &-e^{\lambda x-t}\left(_xD_1^{\alpha}\Big[e^{-\lambda x}((1-x)^3-3(1-x)^4+3(1-x)^5-(1-x)^6)\Big]\right),
  \end{aligned}
  \end{equation}
  the exact solution for \eqref{eq:2-20} reads $u(x,t)=e^{-t}x^{3}(1-x)^{3}$. To compute $f(x,t)$, we use the following formulae,
\begin{equation}
\begin{aligned}
_0D_x^{\alpha}(e^{\lambda x}x^m)=_0D_x^{\alpha}\left(\sum_{n=0}^{\infty}\frac{\lambda^n}{n!}x^{n+m}\right)
=\sum_{n=0}^{\infty}\frac{\lambda^n\Gamma(n+m+1)}{n!\Gamma(n+m-\alpha+1)}x^{n+m-\alpha},
\end{aligned}
\end{equation}
and
\begin{equation}
\begin{aligned}
_xD_1^{\alpha}(e^{\lambda x}(1-x)^m)=&_xD_1^{\alpha}\left(\sum_{n=0}^{\infty}\frac{\lambda^ne^{-\lambda}}{n!}(1-x)^{n+m}\right)\\
=&e^{-\lambda}\sum_{n=0}^{\infty}\frac{\lambda^n\Gamma(n+m+1)}{n!\Gamma(n+m-\alpha+1)}(1-x)^{n+m-\alpha}.
\end{aligned}
\end{equation}
Take $\alpha=1.5\in(1.26,1.71)$ for the tempered fractional diffusion equation \eqref{eq:2-20}. We choose $\gamma_3\in(a_1,a_2)$ for the second-order scheme \eqref{4-1-2}, and $\gamma_4=-0.04,~-0.03,~-0.02$  for the third-order scheme \eqref{4-2-2}, where $\gamma_4=-0.04,\ -0.03\in(a_3,a_4)$ and $\gamma_4=-0.02\notin(a_3,a_4)$ but satisfies $f(\alpha;x)<0$.
Let $\lambda=0.5,\ \ c_r =c_l=0.5,$ and $\tau=10^{-3}.$ Tables \ref{tab:4-2-5} and \ref{tab:4-2-6} show the corresponding $L^2$ errors and the second and third orders of accuracy, respectively.
 \begin{table}[!ht]\centering
\begin{tabular}{c c c c c c c}
\toprule
&\multicolumn{2}{c}{$\gamma_3=-0.04$}    & \multicolumn{2}{c}{$\gamma_3=-0.03$}        & \multicolumn{2}{c}{$\gamma_3=-0.02$}   \\
\cmidrule{2-7}
    h &      \multicolumn{1}{c}{$L^{2}$ Error } &  \multicolumn{1}{c}{Order}   & \multicolumn{1}{c}{$L^{2}$ Error }  & \multicolumn{1}{c}{Order}  &  \multicolumn{1}{c}{$L^{2}$ Error}  & \multicolumn{1}{c}{Order}\\
\midrule
$1/2^4 $&	4.15e-05	&		&	4.19e-05	&		&	4.25e-05	&		\\
$1/2^5 $&	9.55e-06	&	2.12	&	9.98e-06	&	2.07	&	1.04e-05	&	2.03	\\
$1/2^6 $&	2.29e-06	&	2.06	&	2.44e-06	&	2.03	&	2.59e-06	&	2.01	\\
$1/2^7 $&	5.62e-07	&	2.03	&	6.04e-07	&	2.01	&	6.46e-07	&	2.00	\\
$1/2^8 $&	1.39e-07	&	2.01	&	1.50e-07	&	2.01	&	1.62e-07	&	2.00	\\

\bottomrule
\end{tabular}
\caption{$L^2$ errors and orders of accuracy for \eqref{eq:2-20} with the boundary conditions \eqref{eq:2-20a} and the initial value \eqref{eq:2-20b} by the second-order scheme \eqref{4-1-2}.}
\label{tab:4-2-5}
\end{table}

  \begin{table}[!ht]\centering
\begin{tabular}{c c c c c c c}
\toprule
&\multicolumn{2}{c}{$\gamma_4=-0.04$}    & \multicolumn{2}{c}{$\gamma_4=-0.03$}        & \multicolumn{2}{c}{$\gamma_4=-0.02$}   \\
\cmidrule{2-7}
    h &      \multicolumn{1}{c}{$L^{2}$ Error } &  \multicolumn{1}{c}{Order}   & \multicolumn{1}{c}{$L^{2}$ Error }  & \multicolumn{1}{c}{Order}  &  \multicolumn{1}{c}{$L^{2}$ Error}  & \multicolumn{1}{c}{Order}\\
\midrule
$1/2^4 $&	5.41e-05	&		&	5.33e-05	&		&	5.39e-05	&		\\
$1/2^5 $&	7.98e-06	&	2.76	&	7.68e-06	&	2.79	&	7.48e-06	&	2.85	\\
$1/2^6 $&	1.12e-06	&	2.83	&	1.06e-06	&	2.85	&	1.02e-06	&	2.88	\\
$1/2^7 $&	1.52e-07	&	2.88	&	1.43e-07	&	2.89	&	1.35e-07	&	2.91	\\
$1/2^8 $&	2.02e-08	&	2.91	&	1.89e-08	&	2.92	&	1.78e-08	&	2.93	\\

\bottomrule
\end{tabular}
\caption{$L^2$ errors and orders of accuracy for \eqref{eq:2-20} with the boundary conditions \eqref{eq:2-20a} and the initial value \eqref{eq:2-20b} by the third-order scheme \eqref{4-2-2}.}
\label{tab:4-2-6}
\end{table}
\end{enumerate}
\end{example}

\subsection{The tempered fractional Black--Scholes equations}
In this subsection, we construct a third-order scheme for the tempered fractional Black--Scholes equation. These numerical results come without a proof of stability here.
We experimentally show that the developed schemes are robust and accurate, also for a convection-(fractional) diffusion type equation.

Consider the fractional PDE,
\begin{equation}
\label{3-0}
\begin{aligned}
&\frac{\partial u(x,t)}{\partial t}+a \cdot \frac{\partial u(x,t)}{\partial x}
+b\cdot _{B_d}\mathfrak{D}_x^{\alpha,\lambda_1}u(x,t)+d\cdot _x\mathfrak{D}_{B_u}^{\alpha,\lambda_2}u(x,t)
=p u(x,t),\\
\end{aligned}
\end{equation}
where $\alpha\in(1,2)$, the parameters $ b,\ d,\ p,\ \lambda_1$ and $\lambda_2$ are all non-negative. With different values for the parameters $a,\ b,\ d,\ p,\ \lambda_1$ and $\lambda_2$, we find variations for Equation \eqref{3-0}.

We her consider problem \eqref{3-0} with a source term $f(x,t)$ added to test the numerical scheme in the following form,
\begin{equation}
\label{3-1}
\left\{
\begin{aligned}
&\frac{\partial u(x,t)}{\partial t}+a \frac{\partial u(x,t)}{\partial x}
+b_{B_d}\mathfrak{D}_x^{\alpha,\lambda_1}u(x,t)+d_x\mathfrak{D}_{B_u}^{\alpha,\lambda_2}u(x,t)
=p u(x,t)+f(x,t),\\
&(x,t)\in(B_d,B_u)\times(0,T)\\
&u(B_d,t)=0,u(B_u,t)=0,\ t\in(0,T)\\
&u(x,T)=S(x),x\in(B_d,B_u),\ x\in(B_d,B_u).
\end{aligned}
\right.
\end{equation}

Let $t_j=(N-j)\tau$, $ 0\leq t_j\leq T,\ j=0,\dots,N$ and $x_i=B_d+ih,$ $B_d\leq x_i\leq B_u,\ I=0,\dots,M$, where $\tau=T/N$ and $h=(B_u-B_d)/M$. Using the tempered-WSGD operators, $_L\mathcal{D}^{\alpha,\lambda_1}_h=_L\mathcal{D}_{h,-1,0,\alpha-1,1}^{\alpha,\gamma_1,\gamma_2,...,\gamma_4}$ and $_R\mathcal{D}^{\alpha,\lambda_2}_h=_R\mathcal{D}_{h,-1,0,\alpha-1,1}^{\alpha,\gamma_1,\gamma_2,...,\gamma_4}$, for the tempered fractional derivatives, 
and the forth-order scheme for the first-order space derivative, we get the following space discretization for \eqref{3-0}
\begin{equation}
\label{3-3}
\begin{aligned}
&\frac{\partial u_{i}}{\partial t}
+a\frac{8(u_{i+1}-u_{i-1})-(u_{i+2}-u_{i-2})}{12h}
+b\left(_L\mathcal{D}^{\alpha,\lambda_1}_h{u_{i}}\right)
+d\left(_R\mathcal{D}^{\alpha,\lambda_2}_h{u_{i}}\right)\\
&\qquad=pu_{i}+f_i+O(h^3),
\end{aligned}
\end{equation}
where $u_i$ is the solution of \eqref{3-1} when $x=x_i$, and $f_i=f(x_i,t)$.

For \eqref{3-3}, the Crank-Nicolson time discretization reads
\begin{equation}
\label{3-5}
\begin{aligned}
&-\frac{u^{j+1}_{i}-u_{i}^j}{\tau}
+a\frac{8(u^{j+\frac{1}{2}}_{i+1}-u^{j+\frac{1}{2}}_{i-1})-(u^{j+\frac{1}{2}}_{i+2}-u^{j+\frac{1}{2}}_{i-2})}{12h}
+b\left(_L\mathcal{D}^{\alpha,\lambda_1}_h{u^{j+\frac{1}{2}}_{i}}\right)
+d\left(_R\mathcal{D}^{\alpha,\lambda_2}_h{u^{j+\frac{1}{2}}_{i}}\right)\\
&\qquad=pu^{j+\frac{1}{2}}_{i}+f^{j+\frac{1}{2}}_{i}+O(\tau^2+h^3),
\end{aligned}
\end{equation}
where $u_i^j$ is the solution of \eqref{3-1} at the point $(x_i,t_j)$, and $f_i^{j+\frac{1}{2}}=f(x_i,t_{j+\frac{1}{2}})$.

The numerical scheme can now be written as
\begin{equation}
\label{3-7}
\begin{aligned}
&(1+\frac{\tau}{2}p)U_{i}^{j+1}
-\frac{\tau}{2}\left[a\frac{8(U^{j+1}_{i+1}-U^{j+1}_{i-1})-(U^{j+1}_{i+2}-U^{j+1}_{i-2})}{12h}
+b_L\mathcal{D}^{\alpha,\lambda_1}_hU_{i}^{j+1}
+d_R\mathcal{D}^{\alpha,\lambda_2}_hU_{i}^{j+1}
\right]\\
&\qquad=(1-\frac{\tau}{2}p)U_{i}^j
+\frac{\tau}{2}\left[a\frac{8(U^{j}_{i+1}-U^{j}_{i-1})-(U^{j}_{i+2}-U^{j}_{i-2})}{12h}
+b_L\mathcal{D}^{\alpha,\lambda_1}_hU_{i}^{j}
+d_R\mathcal{D}^{\alpha,\lambda_2}_hU_{i}^{j}
\right]-F_i^{j+\frac{1}{2}},
\end{aligned}
\end{equation}
where $U_i^j$ is the solution of the numerical scheme for \eqref{3-1} at point $(t_i,t_j)$, and $F_i^{j+\frac{1}{2}}=\frac{1}{2}(f_i^{j}+f_i^{j+1})$.

\begin{example}\label{e3}
We here consider the following tempered fractional equation,
\begin{equation}
\label{3-18}
\left\{
\begin{aligned}
&\frac{\partial u(x,t)}{\partial t}+a \frac{\partial u(x,t)}{\partial x}
+b\left(_{0}\mathfrak{D}_x^{\alpha,\lambda_1}u(x,t)\right)
=d u(x,t)+f(x,t),\\
&(x,t)\in(0,1)\times(0,T)\\
&u(0,t)=0,u(1,t)=0,\ t\in(0,T)\\
&u(x,T)=S(x),x\in(0,1).
\end{aligned}
\right.
\end{equation}
where
\begin{equation}\nonumber
\begin{aligned}
f=e^{-\lambda x+(T-t)}\Bigg\{&x^3(1-x)(-1-a\lambda-b\lambda^\alpha-p)
+a(3x^2-4x^3)\\
&+b\left(\frac{\Gamma(4)}{\Gamma(4-\alpha)}x^{3-\alpha}
-\frac{\Gamma(5)}{\Gamma(5-\alpha)}x^{4-\alpha}\right)\Bigg\}.
\end{aligned}
\end{equation}
The exact solution of the above equation is $u=e^{T-t}x^3(1-x)^3$. Let the parameters $b=-\frac{1}{2}\sigma^\alpha \sec(\frac{\alpha\pi}{2})$, $a=r-b$ and $d=r$. Take $\alpha=1.6,\ \lambda=1,\ \sigma=0.25,$ $ r=0.05$, and $\tau=10^{-4}$. Table \ref{t3-1} lists the $L^2$ and $L^\infty$ errors and orders of accuracy for Equation \eqref{3-18}, which confirm the desired accuracy with $\gamma_4=-0.5$.
\begin{table}[!ht]\centering
\begin{tabular}{c c c c c }
\toprule
   h &      \multicolumn{1}{c}{$L^{2}$ Error } &  \multicolumn{1}{c}{Order}   & \multicolumn{1}{c}{$L^{\infty}$ Error }  & \multicolumn{1}{c}{Order}  \\
\midrule
$1/2^5\ $&	7.95e-05	&		&	2.58e-04	&		\\
$1/2^6\ $&	1.08e-05	&	2.88	&	3.99e-05	&	2.69	\\
$1/2^7\ $&	1.38e-06	&	2.97	&	5.62e-06	&	2.83	\\
$1/2^8\ $&	1.71e-07	&	3.02	&	7.54e-07	&	2.90	\\
$1/2^9\ $&	2.09e-08	&	3.03	&	9.81e-08	&	2.94	\\
\bottomrule
\end{tabular}
\caption{$L^2$ and $L^\infty$ errors and orders of accuracy for \eqref{3-18} with $\gamma_4=-0.5$.}
\label{t3-1}
\end{table}
\end{example}

\begin{example}\label{e4}
We finally consider the following tempered fractional model
\begin{equation}
\label{3-25}
\left\{
\begin{aligned}
&\frac{\partial u(x,t)}{\partial t}+a \frac{\partial u(x,t)}{\partial x}
+b\left(_{0}\mathfrak{D}_x^{\alpha,\lambda_1}u(x,t)\right)
+c\left(_{0}\mathfrak{D}_x^{\alpha,\lambda_1}u(x,t)\right)
=d u(x,t)+f(x,t),\\
&(x,t)\in(0,1)\times(0,T)\\
&u(0,t)=0,u(1,t)=0,\ t\in(0,T)\\
&u(x,T)=S(x),x\in(0,1),
\end{aligned}
\right.
\end{equation}
where
\begin{equation}\nonumber
\begin{aligned}
f=&-(1+a\lambda_1^\alpha+b\lambda_2^\alpha+p)u(x,t)
+3ae^{T-t}x^2(1-x)^2(1-2x)\\
&+be^{-\lambda_1 x+(T-t)}\left( _0D_x^{\alpha}e^{\lambda_1x}u(x,t)\right)
+ce^{\lambda_2 x+(T-t)}\left(_xD_1^{\alpha}e^{-\lambda_2x}u(x,t)\right).
\end{aligned}
\end{equation}
The exact solution of the above equation is given by $u=e^{-\lambda x+(T-t)}x^3(1-x)$.

Let the parameters $b=c=d=1$ and $a=-0.5$. Take $\alpha=1.8,\ \lambda_1=0.5,\ \lambda_2=1,$ and $\tau=10^{-4}$. Table \ref{t3-2} lists the $L^2$ and $L^\infty$ errors and orders of accuracy for Equation \eqref{3-25}, again confirming the desired accuracy with $\gamma_4=0$.
\begin{table}[!ht]\centering
\begin{tabular}{c c c c c }
\toprule
   h &      \multicolumn{1}{c}{$L^{2}$ Error } &  \multicolumn{1}{c}{Order}   & \multicolumn{1}{c}{$L^{\infty}$ Error }  & \multicolumn{1}{c}{Order}  \\
\midrule
$1/2^5\ $&	2.61e-05	&		&	4.28e-05	&		\\
$1/2^6\ $&	3.57e-06	&	2.87	&	5.52e-06	&	2.95	\\
$1/2^7\ $&	4.74e-07	&	2.91	&	7.08e-07	&	2.96	\\
$1/2^8\ $&	6.19e-08	&	2.94	&	9.06e-08	&	2.97	\\
$1/2^9\ $&	7.99e-09	&	2.95	&	1.16e-08	&	2.97	\\

\bottomrule
\end{tabular}
\caption{$L^2$ and $L^\infty$ errors and orders of accuracy for \eqref{3-25} with $\gamma_4$.}
\label{t3-2}
\end{table}

\end{example}

\section{Conclusion}
In this paper, we presented stability analysis and error estimates for numerical schemes for the tempered fractional diffusion equation. We focussed on the third-order semi-discretized scheme in space and showed error analysis for these fully discrete scheme based on the Crank--Nicolson scheme in time. We also provided the third-order scheme for the tempered fractional Black-Scholes equation. Clearly, the stable numerical schemes proposed in this paper are computationally highly accurate and efficient.

\bibliographystyle{bibft}\it
\bibliography{bibfile}

\end{document}